\documentclass[preprint,11pt,3p]{elsarticle}
\usepackage{verbatim}
\usepackage{amsmath,amsthm,amsfonts,amssymb}
\usepackage{algorithm}
\usepackage{algpseudocode}
\usepackage{enumerate}
\usepackage{graphics}
\usepackage{multicol}
\usepackage{multirow}
\usepackage{cases}
\usepackage[hidelinks]{hyperref}
\usepackage[normalem]{ulem}
\usepackage{xcolor,sectsty}
\definecolor{astral}{RGB}{46,116,181}
\subsectionfont{\color{astral}}
\sectionfont{\color{astral}}
\usepackage{mathtools}
\usepackage{rotating}
\newtheorem{theorem}{Theorem}[section]
\newtheorem{lemma}[theorem]{Lemma}
\newtheorem{definition}[theorem]{Definition}
\newtheorem{example}[theorem]{Example}
\newtheorem{remark}[theorem]{Remark}
\newtheorem{corollary}[theorem]{Corollary}
\newtheorem{proposition}[theorem]{Proposition}

\definecolor{darkslategray}{rgb}{0.18, 0.31, 0.31}
\definecolor{warmblack}{rgb}{0.0, 0.26, 0.26}
\newcommand{\mb}{\mathbb}
\newcommand{\rg}{\mathscr{R}}
\newcommand{\mbc}{\mathbb{C}}
\newcommand{\nl}{\mathscr{N}}
\newcommand{\ind}{{\mathrm {ind}}}
\newcommand{\Ind}{{\mathrm {Ind}}}

\newcommand{\mc}[1]{\mathcal {#1}}
\newcommand{\dg}{{\dagger}}
\newcommand{\m}{{*_{M}}}
\usepackage{graphicx}
\usepackage{mathdots}
\usepackage{MnSymbol}
\usepackage{mathrsfs}
\usepackage{multirow}

\newcommand{\mrk}[1]{{\mathrm{rank}_M}({#1})} 
\newcommand{\trk}[1]{{\mathrm{rank}_t}({#1})}  
\usepackage{hyperref}
\usepackage{comment}
\usepackage{subfigure}
\usepackage[customcolors]{hf-tikz}

\tikzset{style green/.style={
    set fill color=green!50!lime!60,
    set border color=white,
  },
  style cyan/.style={
    set fill color=cyan!90!blue!60,
    set border color=white,
  },
  style orange/.style={
    set fill color=black!10!blue!10,
    set border color=white,
  },
  hor/.style={
    above left offset={-0.15,0.31},
    below right offset={0.15,-0.125},
    #1
  },
  ver/.style={
    above left offset={-0.1,0.3},
    below right offset={0.15,-0.15},
    #1
  }
}

\usepackage{tikz,xcolor}
\usetikzlibrary{arrows.meta}
\tikzset{%
  >={Latex[width=2mm,length=2mm]},
            base/.style = {rectangle, rounded corners, draw=black,
                           minimum width=4cm, minimum height=1cm,
                           text centered, font=\sffamily},
  activityStarts/.style = {base, fill=blue!25},
       startstop/.style = {base, fill=red!25},
    Rati44/.style = {base, fill=green!25},
    Rati55/.style = {base, fill=green!25},
    Rati56/.style = {base, fill=green!25},
    Rati60/.style = {base, fill=red!25},
         process/.style = {base, minimum width=2.5cm, fill=orange!15,
                           font=\ttfamily},
}

\definecolor{lime}{HTML}{A6CE39}
\definecolor{lightblue}{rgb}{0.0, 0.0, 0.5}
\DeclareRobustCommand{\orcidicon}{%
	\begin{tikzpicture}
	\draw[lime, fill=lime] (0,0)
	circle [radius=0.16]
	node[white] {{\fontfamily{qag}\selectfont \tiny ID}};
	\draw[white, fill=white] (-0.0625,0.095)
	circle [radius=0.007];
	\end{tikzpicture}
	\hspace{-2mm}
}
\foreach \x in {A, ..., Z}{%
	\expandafter\xdef\csname orcid\x\endcsname{\noexpand\href{https://orcid.org/\csname orcidauthor\x\endcsname}{\noexpand\orcidicon}}
}

\usepackage{scalerel}
\usepackage{stackengine,wasysym}
\newcommand\reallywidetilde[1]{\ThisStyle{%
  \setbox0=\hbox{$\SavedStyle#1$}%
  \stackengine{-.1\LMpt}{$\SavedStyle#1$}{%
    \stretchto{\scaleto{\SavedStyle\mkern.2mu\AC}{.30\wd0}}{.4\ht0}%
  }{O}{c}{F}{T}{S}%
}}

\def\wtilde#1{
  \reallywidetilde{#1}}

\begin{document}

\begin{frontmatter}

\title{
Computation of $M$-QDR decomposition of tensors and applications 
}
\author{Krushnachandra Panigrahy$^{a,1}$, Biswarup Karmakar$^{a,2}$, Jajati Keshari Sahoo$^{b}$\footnote{Corresponding Author}, Ratikanta Behera$^{a,3}$, Ram N. Mohapatra$^{c}$} 
\vspace{.3cm}

\address{  
$^{a}$Department of Computational and Data Sciences, Indian Institute of Science Bangalore, 560012, India.\\
                        \textit{E-mail$^{2}$}: \texttt{krushnachand@iisc.ac.in}\\
                        \textit{E-mail$^{4}$}: \texttt{ratikanta@iisc.ac.in}\\
                         \textit{E-mail$^{1}$}: \texttt{biswarupk@iisc.ac.in}\\
                        \vspace{.3cm}
      $^{b}$Department of Mathematics, BITS Pilani, K.K. Birla Goa Campus, Goa, India\\
        \textit{E-mail}: \texttt{jksahoo@goa.bits-pilani.ac.in}\\
        $^{c}${Department of Mathematics, University of Central Florida, Orlando, Florida, USA\\ \textit{E-mail\,$^{c}$}: \texttt{ram.mohapatra@ucf.edu}}\\
}
    
\begin{abstract}
The theory and computation of tensors with different tensor products play increasingly important roles in scientific computing and machine learning. Different products aim to preserve different algebraic properties from the matrix algebra, and the choice of tensor product determines the algorithms that can be directly applied. This study introduced a novel full-rank decomposition and $M$-$\mc{QDR}$ decomposition for third-order tensors based on $M$-product. Then, we designed algorithms for computing these two decompositions along with the Moore–Penrose inverse, and outer inverse of the tensors. In support of these theoretical results, a few numerical examples were discussed. In addition, we  derive exact expressions for the outer inverses of tensors using symbolic tensor (tensors with polynomial entries) computation. We designed efficient algorithms to compute the Moore-Penrose inverse of symbolic tensors. The prowess of the proposed $M$-$\mc{QDR}$ decomposition for third-order tensors is applied to compress lossy color images.

\end{abstract}

\end{frontmatter}

\section{Introduction and Preliminaries}
Most researchers have encountered phrase tensor computations and applications \cite{kolda09, Liang, Qi07, qi2017}. It is a higher-dimensional generalization of a matrix and vector that has become extremely popular in many advanced mathematical and scientific fields (e.g., such as numerical multilinear algebra \cite{Qi07}, image processing \cite{martin13,soltani16}, computer vision \cite{hao2013,hu16}, face recognition \cite{kilmer13}, and robust tensor  Principal Component Analysis (PCA) \cite{liu18}). The field of tensor computations has  grown significantly, with many powerful software tools (e.g., MATLAB, TensorFlow, PyTorch and NumPy) available in the public domain. These tools have significantly accelerated research and applications in the aforementioned scientific fields, including machine learning \cite{SidLaP17} and scientific computing \cite{qi2017, tarzanagh2018}, where tensor techniques are increasingly important. In recent years, a rapidly growing body of literature has focused on tensor computations, particularly in the context of different tensor products, including  $t$-product \cite{Braman10, kilmer13, martin13}, $M$-product \cite{kernfeld2015}, Einstein products \cite{ein, BraliNT13, Ji2018Drazin,Ji2020outer}, $n$-mode products \cite{BaderKolda07}, cosine transform products ($c-$product \cite{kernfeld2015}), and Shao's general product \cite{Shao13}. %However, the Einstein product and $t$-product are used most often among these products. 

In contrast, tensor decomposition provides a robust framework for extracting meaningful information from complex datasets, often revealing patterns and relationships that are not apparent in traditional matrix-based analyses \cite{kolda09}. Following foundational work on different tensor products, the decomposition of third-order tensors has emerged as a significant area of research in multilinear algebra and its applications. Verity of tensor decomposition (e.g., ${QR}$  decomposition \cite{kilmer13, QRBehera}, singular value decomposition \cite{kilmer11, behera2022, BraliNT13}, CUR decomposition \cite{tarzanagh2018}, ${URV}$ decomposition \cite{che22} is discussed \cite{hao2013, GenweiEin16} based on $t$-product and the Einstein product. Additionally, each  decomposition has its own strengths and is suited to different types of problems. The choice between them often depends on the specific requirements of the application, such as computational efficiency, interpretability, or accuracy of the approximation.

The main aim of this study was to discuss QDR decomposition based on the M-product.  $M$-product is a family of tensor-tensor products with given fixed  invertible linear transforms, which generalize the previously established $t$-product \cite{kilmer11} and c-product \cite{kernfeld2015}. Following the $M$-Product structure, we introduce two important tensor decompositions: full-rank decomposition and $M$-$\mc{QDR}$ decomposition. Following these two decompositions, we designed effective algorithms for Moore-Penrose inverses and outer inverses of tensors.

It is well known that computations involving square-root entries are challenging for algebraic computations and symbolic polynomial computation of tensors. Therefore, scientists prefer to design algorithms; such as such as polynomial rings or algebraic extensions of rational numbers,  where square roots are avoided or carefully managed. It is well known that the square root of some tensor polynomials often occurs when generating tensor QR-factorization. Generating expressions that include square roots can be avoided  using  tensor QDR-decomposition. In addition, our proposed decomposition in symbolic computation offers several advantages, including the fact that symbolic QDR decomposition provides exact factorization of tensors without numerical approximations and preserves the algebraic structure. At the same time, it avoids the numerical errors associated with floating-point arithmetic. Specifically, it is helpful for ill-conditioned tensors where numerical methods may  fail. In contrast, symbolic QDR decomposition is more efficient for tensors with many zero or symbolic zero entries. This avoids unnecessary computations on the zero elements. This decomposition is efficiently implemented in Mathematica, Maple, or SymPy systems. 

The volume of visual data generated daily continues to surge exponentially,  driven by high-resolution cameras, advanced medical imaging, satellite observations, and the ubiquity of multimedia content. Thus, the substantial data requirements for high-quality digital images present challenges in terms of the transmission speed and storage costs. Consequently, there is a pressing need for efficient data compression techniques to reduce the file size of these images while maintaining their visual integrity. We used the $M$-product-based QDR decomposition algorithm for lossy image compression. Lossy image compression is a technique for  reducing digital image file sizes by intentionally discarding  image data. This technique balances the trade-off between image quality and file size, thereby achieving significant reductions in storage space and transmission bandwidth at the cost of  image fidelity. The main contributions of the manuscript are as follows:

\begin{itemize}
    \item[(1)] We introduce full-rank decomposition (FRD) and $M$-$\mc{QDR}$ decomposition based on $M$-product, along with powerful algorithms for computation.  
     \item[(2)] We compute the generalized inverses of the tensors based on the FRD in the framework of the $M$-product. 
    \item[(3)]  We designed an effective algorithm for Moore-Penrose inverses using the FRD and the outer inverse of the tensors.
\item[(4)] We discuss the application of tensor-based QDR decomposition for lossy image compression.
    \item[(5)] We compute the outer inverses of the tensors based on $M$-$\mc{QDR}$ decomposition via symbolic computation based on the $M$-product. 
\end{itemize}

The remainder of the paper is organized as follows. In Section \ref{sec:pre}, we present some notations and definitions of third-order tensors \cite{kilmer13, kilmer11, martin13, soto23}, which help prove the main results. Section \ref{Sec2} presents the FRD of a tensor in the framework of $M$-product. In addition to this, we designed an algorithm for computing FRD and the generalized inverse of a tensor. In Section \ref{sec:tqdr}, we propose $M$-QDR decomposition of a tensor via $M$-product with the help of the FRD for computing the outer inverse of tensors. We demonstrate the application of this decomposition to color lossy image compression. In Section \ref{sec:outsy} we present the computation of outer inverses via symbolic computation. Finally, we discuss concluding remarks in Section \ref{sec:con}.

\subsection{Preliminaries}\label{sec:pre}
We denote the frontal slices of the tensor $\mathcal{A} \in \mathbb{C}^{m \times n \times p}$ as $\mathcal{A}^{(i)}=\mathcal{A}(:,:,i)$ for $i=1,2,\cdots, p$. Similarly, the tube fibers of $\mathcal{A}$ are denoted as $\mathcal{A}(i, j,:)$, $\mathcal{A}(i,:,k)$, and $\mathcal{A}(:,j,k)$.

\begin{definition}\cite{kolda09}
    The $3$-mode product between a tensor ${\mc{A}} \in \mb{C}^{m \times n \times p}$ and  a matrix $U\in\mb{C}^{ q \times p}$ is denoted by $\mc{A}\times_3 U$ and is defined as\[(\mc{A}\times_3 U)(i_1,i_2,k)=\sum_{i_3=1}^p \mc{A}(i_1,i_2,i_3)B(k,i_3)\quad \text{where}\quad  k\in \{1,\ldots,q\}.\]
\end{definition}
\noindent Let ${\mc{A}} \in \mb{C}^{m \times n \times p}$ and ${\mc{B}} \in \mb{C}^{n \times q \times p}$, then the face-wise product \cite{kernfeld2015}, denoted by $\mc{A}\Delta\mc{B}$, is defined as \[
    (\mc{A}\Delta\mc{B})^{(i)}=\mc{A}^{(i)}\mc{B}^{(i)}.
    \]
Based on $3$-mode and face-wise product, Kernfeld {\it et al.} \cite{kernfeld2015} introduced the $M$-product for tensors.
\begin{definition}\cite{kernfeld2015}
   Let $\mc{A},~\mc{B} \in\mb{C}^{m\times n\times p}$. Consider an invertible matrix $M\in\mb{C}^{p \times p}$. 
    Then the $M$-product of $\mc{A}$ and $\mc{B}$ is denoted by $\mc{A}\m\mc{B}$, and is defined as 
    \begin{equation*}
        \mc{A}\m\mc{B}=(\Tilde{\mc{A}}\Delta\Tilde{\mc{B}})\times_3 M^{-1}, \textnormal{~where~} \Tilde{\mc{A}}=\mc{A}\times_3 M \textnormal{~and~} \Tilde{\mc{B}}=\mc{B}\times_3 M.
    \end{equation*}
    
\end{definition}
Notice that from here, we will consider the matrix $M$ to be invertible and the tilde ($\Tilde{.}$) notation when multiplied by $M$, that is $(\Tilde{.})=(.)\times_3M$. Note that the $M$-product is identical to the $t$-product \cite{kilmer11} when considering $M$ as the unnormalized DFT matrix. Further, the $M-$product is identical to the $c-$product \cite{kernfeld2015} when considering the discrete cosine transform (DCT) matrix. See \cite{sahoo2024computation} for more details about the matrix constructions. 

The following result immediately follows from the definition of $3$-mode product, $M$-product, and face-wise product.
\begin{proposition}\label{prop1.11}
Let $\mc{A},~\mc{B} \in\mb{C}^{m\times n\times p}$ and $M\in\mb{C}^{p \times p}$. Then 
    \begin{enumerate}
        \item[(i)] $(\mc{A}\m\mc{B})\times_3M=\Tilde{\mc{A}}\Delta\Tilde{\mc{B}}=(\mc{A}\times_3M)\Delta(\mc{B}\times_3M)$.
        \item[(ii)] $(\Tilde{\mc{A}})^*=\wtilde{\mc{A}^*}=\mc{A}^*\times_3M$.
        \item[(iii)] $\wtilde{\mc{A}+\mc{B}}=\Tilde{\mc{A}}+\Tilde{\mc{B}}$.
        \item[(iv)] $\left(\wtilde{\mc{A}\m\mc{B}}\right)^{-1}=(\Tilde{\mc{B}})^{-1}\Delta(\Tilde{\mc{A}})^{-1}$.
        \item[(v)] $\left({\mc{A}\m\mc{B}}\right)^{-1}=\mc{B}^{-1}\m 
\mc{A}^{-1}$.
     \end{enumerate}
\end{proposition}

Let $\Tilde{\mc{A}}^{(k)}=\Tilde{\mc{A}}(:,:,k)$ denotes the $k^{th}$ frontal slice of $\Tilde{\mc{A}}$ and the rank of  $k^{th}$-frontal slice $\Tilde{\mc{A}}^{(k)}$ be $r_k$. Then the ordered $p$-tuple $(r_1,r_2,\ldots,r_p)$, is known as the multirank of $\mc{A}$ \cite{soto23} and we denote it by $\mrk{\mc{A}}$.  For $\mc{A}\in\mb C^{m\times n\times p}$, we further denote $\mrk{\mc{A}}={\mathbf{r}}=(r,r,\ldots,r)$ and $\trk{\mc{A}}=\displaystyle\max_{1\leq i\leq p}\{r_i\}$, called the tubal rank of $\mc{A}$. If the frontal slices of a tensor $\Tilde{\mc{A}}$, are identity matrices, then the respective tensor $\mc{A}\in \mbc^{m\times m\times p}$ is called an identity tensor with respect $M$-product and is denoted by $\mc{I}_{mmp}$. Similarly, a  tensor $\mc{A}\in \mb{C}^{n \times n \times p}$ is invertible if all frontal slices of $\Tilde{\mc{A}}$ are invertible.  A third-order tensor $\mc{A}$ is called an upper triangular if the frontal slices of $\Tilde{\mc{A}}$ are upper triangular matrices. Similarly, we can define lower triangular tensors and diagonal tensors. In addition, $\mb{C}(x)^{m\times n\times p}$  denotes the set of tensors over the field of rational fractions $\mb{C}(x)$.  The range space of $\mc{A}$  is denoted by $\mathscr{R}(\mc{A})$ and the null space of $\mc{A}$ is denoted by $\mathscr{N}(\mc{A})$. 
Now we define:
%the $\mathscr{R}(\mc{A})$ and $\mathscr{N}(\mc{A})$ as follows: 
 \begin{eqnarray*}
 \mathscr{R}(\mc{A})&=&\{\mc{A}\m\mc{X}\mid\mc{X}\in\Omega \}.~~\mathscr{N}(\mc{A})=\{\mc{X}\in\Omega \mid \mc{A}\m\mc{X}=\mc{O}\}. 
 \end{eqnarray*}
%Recall that, for a tensor $\mc{A}\in\mb{R}^{m\times n\times p}$, the range space of $\mc{A}$ \cite{sahoo2024} is denoted by $\mathscr{R}(\mc{A})$ and is defined as $\mathscr{R}(\mc{A})=\{\mc{Y}\in\mb{R}^{m\times 1\times p}\mid \mc{A}\m\mc{X}=\mc{Y}, \mc{X}\in\mb{R}^{n\times 1\times p}\}.$ The null space of $\mc{A}$ is denoted by $\mathscr{N}(\mc{A})$ and is defined as $\mathscr{N}(\mc{A})=\{\mc{X}\in\mb{R}^{n\times 1\times p}\mid \mc{A}\m\mc{X}=\mc{O}\}$. 

Next, we recall the following result, which was proven in \cite{sahoo2024computation}.
\begin{lemma} \cite[Proposition 2.10]{sahoo2024computation} \label{rgnl}
Let $M\in\mbc^{p\times p}$ and $\mc{A},~\mc{B} \in \mbc^{m\times n\times p}$. 
 Then $\rg(\mc{A})\subseteq\rg(\mc{B})$ if and only if $\mc{A}=\mc{B}\m\mc{T}$ for some $\mc{T}\in \mbc^{n\times n\times p}$. 
\end{lemma}
If $\mc{Z}$ satisfies $\mc{Z}*\mc{A}*\mc{Z} = \mc{Z}$ $\rg(\mc{X})=\rg(\mc{B})$, and  $\nl(\mc{X}) = \nl(\mc{C})$ then $\mc{Z}$ is called the outer inverse of $\mc{A}$  (with a given null and range space), and denoted by $\mc{A}^{{(2)}}_{\rg(\mc{B}),\nl(\mc{C})}$.  We now recall the definition of the Moore-Penrose inverse of tensors based on $M$-product 
\begin{definition}\cite{jin2023}\label{MPdef}
For any tensor $\mc{A} \in \mbc^{m\times n\times p}$  consider the following equations in $\mc{Z} \in \mbc^{n\times m\times p}:$
\begin{eqnarray*}
(1)~\mc{A}*\mc{Z}*\mc{A} = \mc{A},~
(2)~\mc{Z}*\mc{A}*\mc{Z} = \mc{Z},~
(3)~(\mc{A}*\mc{Z})^T = \mc{A}*\mc{Z},~
(4)~(\mc{Z}*\mc{A})^T = \mc{Z}*\mc{A}.
\end{eqnarray*}
Then $\mc{Z}$ is called the Moore-Penrose inverse of $\mc{A}$ if it satisfies all four conditions, which is denoted by $\mc{A}^{\dagger}.$
\end{definition}
  The multi-index and tubal index of a tensor are defined below.
\begin{definition}
The multi index of $\mc{A} \in \mbc^{m\times m\times p}$ induced by  $M\in\mbc^{p\times p}$ is denote by $\Ind_M(\mc{A})$ and is defined as
 $\Ind_M(\mc{A})=(k_1,\ldots,k_p)$ where $k_i=\ind(\tilde{\mc{A}}(:,:,i))$, $i=1,2,\ldots, p$.
 The standard notation $\ind(A)$ represents the minimal non-negative integer $k$ determining rank-invariant powers $\mathrm{rank}(A^{k})=\mathrm{rank}(A^{k+1})$.
 The tubal index $k$ is defined as $k:=\max_{1\leq i\leq p}k_i$.
 \end{definition}
\begin{definition}\cite{sahoo2024computation}\label{DRdef}
Let $M\in\mbc^{p\times p}$ and $\mc{A} \in \mbc^{m\times m\times p}$ with a tubal index $k$. If a tensor $Z$ satisfying 
The Drazin inverse $\mc{A}^D$ of $\mc{A}$ is defined as a unique tensor $\mc{Z}\in \mbc^{m\times m\times p}$ satisfying 
\[\mc{Z}\m\mc{A}\m\mc{Z}=\mc{Z}, ~\mc{A}\m\mc{Z}=\mc{Z}\m\mc{A} \mbox{ and } \mc{Z}\m\mc{A}^{k+1}=\mc{A}^k,\] the $\mc{Z}$ is called the Drazin inverse $\mc{A}$ and is denoted by $\mc{A}^D$.
\end{definition}

\section{Full-rank decomposition based on $M$-product }\label{Sec2}

\begin{definition}\label{def:frd}
    Let $M\in\mb C^{p\times p}$ and $\mc{A}\in\mb{C}^{m\times n\times p}$ with $\trk{\mc{A}}=r$. If there exist tensors $\mc{S}\in\mb{C}^{m\times r\times p}$ and $\mc{T}\in\mb{C}^{r\times n\times p}$ such that $\trk{\mc{S}}=r=\trk{\mc{T}}$ and $\mc{A}=\mc{S}\m\mc{T}$ then we call $\mc{A}=\mc{S}\m\mc{T}$ is a FRD of $\mc{A}$.
\end{definition}
The following algorithm is proposed to compute the FRD for any tensor $\mc{A}\in\mb{C}^{m\times n\times p}$.
\begin{algorithm}[H]
 \caption{FRD of $\mc{A}\in\mb{C}^{m\times n\times p}$} \label{alg:fr}
\begin{algorithmic}[1]
\State {\bf Input} $\mc{A}\in \mb{C}^{m\times n\times p}$ and $M\in\mb{R}^{p\times p}$
\State {\bf Compute } $\Tilde{\mc{A}}=\mc{A}\times_{3}M$ and $r=\trk{\mc{A}}$
\For{$i=1:p$}
\State $[F,G]=\mbox{FRD}(\Tilde{\mc{A}}(:,:,i))$
\State {\bf Compute} $r_1=\mathrm{rank}(\Tilde{\mc{A}}(:,:,i))$
\If{$r_1=r$}
\State $\Tilde{\mc{S}}(:,:,i)=F$ and $\Tilde{\mc{T}}(:,:,i)=G$
\Else
\State $\Tilde{\mc{S}}(:,:,i)=\begin{bmatrix}F&\mathbf{0}_{m,r-r_1}\end{bmatrix}$; \State $\Tilde{\mc{T}}(:,:,i)=\begin{bmatrix}G\\\mathbf{0}_{r-r_1,n}\end{bmatrix}$
\EndIf
\EndFor
\State {\bf Compute} $\mc{S}=\Tilde{\mc{S}}\times_{3}M^{-1}$ and $\mc{T}=\Tilde{\mc{T}}\times_{3}M^{-1}$
\State \Return $\mc{S},\mc{T}$.
 \end{algorithmic}
\end{algorithm}
\begin{example}\rm\label{exm:cntrfg}
    Let $\mc{A}\in\mb{C}^{2\times 3\times 2}$
with frontal slices
\[
\mc{A}(:,:,1) =
    \begin{pmatrix}
   0 & 1 &0\\
       0 & 1 &1
    \end{pmatrix},~
\mc{A}(:,:,2) =
    \begin{pmatrix}
      1 & 0 &1\\
     0 & 1 &1
    \end{pmatrix},~\textnormal{and}~
    {M} =
    \begin{pmatrix}
     1 &  -1\\
     0 & 1
    \end{pmatrix}.\]
We can compute  $\mathrm{rank}_{t}(\mc{A})=2=\max\{1,2\}$ and by Algorithm  \ref{alg:fr}, we obtain   $\mc{S}\in\mb{C}^{2\times 2\times 2}$ and $\mc{T}\in\mb{R}^{2\times 3\times 2}$ with the respective frontal slices are given by 
\[\mc{S}(:,:,1) =
    \begin{pmatrix}
   0.3855   & 0.7071\\
	   -0.9306 & -0.7071
    \end{pmatrix},~
\mc{S}(:,:,2) =
    \begin{pmatrix}
    -0.9306  &  0.7071\\
	   -0.9306  & -0.7071
    \end{pmatrix}, \]
 \[  \mc{T}(:,:,1) =
    \begin{pmatrix}
  -1.2971   & 0.2226 &  -1.8344\\
	    0.7071  & -0.7071  & -0.0000
    \end{pmatrix},~
\mc{T}(:,:,2) =
    \begin{pmatrix}
    -0.5373   &-0.5373  & -1.0746\\
	    0.7071  & -0.7071 &  -0.0000
    \end{pmatrix}.\]
  Moreover, $\mc{S}^{T}\m\mc{S}\in\mb{R}^{2\times 2\times 2}$ and $\mc{T}\m\mc{T}^{T}\in\mb{R}^{2\times 2\times 2}$ with respective frontal slices are
  \[\mc{S}^{T}\m\mc{S}(:,:,1)=\begin{pmatrix}
      3.4641    &     0\\
	         0   & 1
  \end{pmatrix}=\mc{T}\m\mc{T}^{T}(:,:,1),~\mc{S}^{T}\m\mc{S}(:,:,2)=\begin{pmatrix}
      1.7321   &     0\\
	         0   & 1
  \end{pmatrix}=\mc{T}\m\mc{T}^{T}(:,:,2).\]      
 Clearly both $\mc{S}^{T}\m\mc{S}$ and $\mc{T}\m\mc{T}^{T}$ are not invertible with respect to the matrix $M$ since 
 \[[(\mc{S}^{T}\m\mc{S})\times_3M](:,:,1)=\begin{pmatrix}
      1.7321      &     0\\
	         0   & 0
  \end{pmatrix}=[(\mc{T}\m\mc{T}^{T})\times_3M](:,:,1),\] which are not invertible.
\end{example}
\begin{remark}\label{rmk2.5}
In Example \ref{exm:cntrfg}, we observe that when we perform a FRD of a third-order tensor $\mathcal{A}=\mathcal{S} \times_{1} \mathcal{T}$, the matrices $\mathcal{S}^{T} \mathcal{S}$ and $\mathcal{T} \mathcal{T}^{T}$ are not invertible under the $M$-product. However, this invertibility property holds for the matrix FRD.
\end{remark}
In light of Remark \ref{rmk2.5}, the following question naturally arises:
\begin{itemize}
    \item \textbf{Question:} Is there any third-order tensor $\mc{A}$ 
    with FRD $\mc{A}=\mc{S}\m\mc{T}$ such that
$\mc{S}^{T} \m \mc{S}$ and $\mc{T} \m \mc{T}^{T}$ are invertible? 
\end{itemize}
Absolutely, the answer is positive, and we categorize such tensors in the following result.
\begin{lemma}\label{lma:invrt}
    Let $M\in\mb{C}^{p\times p}$ and $\mc{A}\in\mb{C}^{m\times n\times p}$ with $\mrk{\mc{A}}=\mathbf{r}$. If $\mc{A}=\mc{S}\m\mc{T}$ is a FRD of $\mc{A}$,  then $\mc{S}^{*} \m \mc{S}$ and $\mc{T} \m \mc{T}^{*}$ are invertible. 
\end{lemma}
\begin{proof}
Let   $\mrk{\mc{A}}=\mathbf{r}$. Then  $\mathrm{rank}(\Tilde{\mc{A}}(:,:,k))=r=\mathrm{rank}(\Tilde{\mc{S}}(:,:,k))=\mathrm{rank}(\Tilde{\mc{T}}(:,:,k))$ for each $k$ where $k=1,2,\ldots,p$. From $\mathrm{rank}(\Tilde{\mc{S}}^{
*}(:,:,k)\Tilde{\mc{S}}(:,:,k))=\mathrm{rank}(\Tilde{\mc{S}}(:,:,k))=r$, it follows that $\Tilde{\mc{S}}^{*}(:,:,k)\Tilde{\mc{S}}(:,:,k)$ is invertible for each $k$, $k=1,2,\ldots, p$. Thus by Proposition \ref{prop1.11} (i), $\wtilde{\mc{S}^{*} \m \mc{S}}(:,:,k)$ is invertible for each $k$.  Hence $\mc{S}^{*} \m \mc{S}$  is invertible. Similarly, we can show the invertibility of $\mc{T} \m \mc{T}^{T}$.
\end{proof}
\subsection{Computation of generalized inverses via FRD}
In this subsection, we discuss the computation of the Moore-Penrose inverse and, in a more general setting, the outer inverse via FRD.
\begin{theorem}
   Let $M\in\mb{C}^{p\times p}$ and $\mc{A}\in\mb{C}^{m\times n\times p}$ with $\mrk{\mc{A}}=\mathbf{r}$. If $\mc{A}=\mc{S}\m\mc{T}$ is a FRD of $\mc{A}$ then  $\mc{A}^{\dg}=\mc{T}^{*}\m(\mc{S}^{*}\m\mc{A}\m\mc{T}^{*})^{-1}\m\mc{S}^{*}$.
\end{theorem}
\begin{proof}
Let $\mrk{\mc{A}}=\mathbf{r}$ and $\mc{A}=\mc{S}\m\mc{T}$ is a FRD of $\mc{A}$. Then by Proposition \ref{prop1.11} (v) and Lemma \ref{lma:invrt}, we obtain $\mc{S}^{*}\m\mc{A}\m\mc{T}^{*}$ is invertible and 
\[(\mc{S}^{*}\m\mc{A}\m\mc{T}^{*})^{-1}=(\mc{T}\m\mc{T}^{*})^{-1}\m(\mc{S}^{*}\m\mc{S})^{-1}.\]
Let $\mc{Z}=\mc{T}^*\m(\mc{T}\m\mc{T}^{*})^{-1}\m(\mc{S}^{*}\m\mc{S})^{-1}\m\mc{S}^*$. Then, through simplification, we can verify the following:
\begin{center}
$\mc{A}\m\mc{Z}\m\mc{A}=\mc{A}$, $\mc{Z}\m\mc{A}\m\mc{Z}=\mc{Z}$, $(\mc{A}\m\mc{Z})^{T}=\mc{A}\m\mc{Z}$ and $(\mc{Z}\m\mc{A})^{T}=\mc{Z}\m\mc{A}$.     
\end{center}
\end{proof}
In the following Algorithm \ref{alg:frmpi}, we propose an alternative algorithm to compute the Moore-Penrose of a tensor when the ranks of the frontal slices of $\Tilde{\mc{A}}$ are different.
\begin{algorithm}[H]
 \caption{ Computation of the $\mc{A}^{\dagger}$ based on FRD} \label{alg:frmpi}
\begin{algorithmic}[1]
\State {\bf Input} $\mc{A}\in \mb{C}^{m\times n\times p}$ and $M\in\mb{C}^{p\times p}$; 
\State {\bf Compute } $\Tilde{\mc{A}}=\mc{A}\times_{3}M$;
\For{$i=1:p$}
\State $[F,G]=\mbox{FRD}(\Tilde{\mc{A}}(:,:,i))$;
\State $\Tilde{\mc{X}}(:,:,i)=G^{*}(F^{*} AG^{*})^{-1}F^{*}$;
\EndFor
\State $\mc{X}=\Tilde{\mc{X}}\times_{3}M^{-1}$;
\State \Return $\mc{A}^{\dg}=\mc{X}$.
 \end{algorithmic}
\end{algorithm}
The following numerical example is worked out to illustrate Algorithm  \ref{alg:frmpi} for computing the Moore-Penrose inverse.
\begin{example}\rm
  Let  $\mc{A}\in\mb{C}^{3\times 3\times 3}$ with frontal slices are  
\[\mc{A}(:,:,1)=\begin{pmatrix}
    0   &  1 &    0\\
	    -1  &   0   &  1\\
	    -1   &  0   &  1
\end{pmatrix},~\mc{A}(:,:,2)=\begin{pmatrix}
   1   &  1  &   1\\
	     1  &   1   &  2\\
	    -1   &  0   & -1
\end{pmatrix},~\mc{A}(:,:,3)=\begin{pmatrix}
    1    & 2  &   1\\
	     2  &   1  &   1\\
	     1   &  1  &   1
\end{pmatrix},\mbox{ and }M=\begin{pmatrix}
     1  &  -1 &    1\\
     0   &  1  &   1\\
     0   &  0   &  1
\end{pmatrix}.\]
Clearly $\mrk{\mc{A}}=(2,2,3)$. Using Algorithm \ref{alg:frmpi}, we obtain $\mc{X}=\mc{A}^{\dagger}$ with frontal slices are given by
\[\mc{X}(:,:,1)=\begin{pmatrix}
    -0.1553   &  -1.7632  &    1.9421\\
	   -1.1053   &  -0.2632  &    2.3421\\
	    1.7447   &   2.2368   &  -5.8579
\end{pmatrix},~\mc{X}(:,:,2)=\begin{pmatrix}
   -0.1053   &  -0.7632   &   0.8421\\
	   -0.6053   &  -0.2632  &    1.3421\\
	    0.8947   &   1.2368   &  -3.1579
\end{pmatrix},\]
and $\mc{X}(:,:,3)=\begin{pmatrix}
    0 & 1&  -1\\
    1&  0 & -1\\
    -1 & -1&  3
\end{pmatrix}$. We can further evaluate the errors associated with the Penrose equations as follows.
\[\|\mc{A}\m\mc{X}\m\mc{A}-\mc{A}\|_{F}=4.8074e^{-15},~~\|\mc{X}\m\mc{A}\m\mc{X}-\mc{X}\|_{F}=2.4401e^{-15},\]
\[\|\mc{A}\m\mc{X}-(\mc{A}\m\mc{X})^{T}\|_{F}=5.3314e^{-15},~~\|\mc{X}\m\mc{A}-(\mc{X}\m\mc{A})^{T}\|_{F}=2.0244e^{-15}.\]
\end{example}
\begin{remark}    
Algorithm \ref{alg:frmpi} is very useful to compute the Moore-Penrose inverse of a tensor when $\mrk{\mc{A}}=(r_{1},r_{2},\ldots, r_{p})$ with $r_i\neq r_j$ for some $i\neq j$. However, the drawback is that it  can not produce $\mc{S}$ and $\mc{T}$ explicitly due to incompatibility in all the frontal slices dimensions.
\end{remark}

\begin{lemma}\label{lma:rannul}
 Let $M\in\mb{C}^{p\times p}$, $\mc{S}\in\mb{C}^{m\times r\times p}$ and $\mc{T}\in\mb{C}^{r\times n\times p}$. If  $\mrk{\mc{S}}=\mathbf{r}=\mrk{\mc{T}}$, then $\mathscr{R}(\mc{S}\m\mc{T})=\mathscr{R}(\mc{S})$ and $\mathscr{N}(\mc{S}\m\mc{T})=\mathscr{N}(\mc{T})$.
\end{lemma}
\begin{proof}
   Let $\Tilde{\mc{S}}=\mc{S}\times_3 M$ and $\Tilde{\mc{T}}=\mc{T}\times_3 M$.  To show $\mathscr{R}(\mc{S}\m\mc{T})=\mathscr{R}(\mc{S})$ it is sufficient to show that $\mathscr{R}(\Tilde{\mc{S}}(:,:,k)\Tilde{\mc{T}}(:,:,k))= \mathscr{R}(\Tilde{\mc{S}}(:,:,k))$ for $k=1,\ldots,p$. Since $\mathscr{R}(\Tilde{\mc{S}}(:,:,k)\Tilde{\mc{T}}(:,:,k))\subseteq \mathscr{R}(\Tilde{\mc{S}}(:,:,k))$, so the hypothesis $\textnormal{multi-rank}(\mc{S})=r$ implies $\mathscr{R}(\Tilde{\mc{S}}(:,:,k)\Tilde{\mc{T}}(:,:,k))= \mathscr{R}(\Tilde{\mc{S}}(:,:,k))$. Therefore, $\mathscr{R}(\mc{S}\m\mc{T})=\mathscr{R}(\mc{S})$. Similarly, it can be shown that $\mathscr{N}(\mc{S}\m\mc{T})=\mathscr{N}(\mc{T})$.
\end{proof}
Lemma \ref{lma:rannul} establishes the link between the FRD of a tensor with the null space and range space. 
\begin{corollary}\label{cor:rngnul}
    Let $M\in\mb{C}^{p\times p}$ and $\mc{A}\in\mb{C}^{m\times n\times p}$ with $\mrk{\mc{A}}=\mathbf{r}$. If $\mc{A}=\mc{S}\m\mc{T}$ is a FRD of $\mc{A}$, then $\mathscr{R}(\mc{A})=\mathscr{R}(\mc{S})$ and $\mathscr{N}(\mc{A})=\mathscr{N}(\mc{T})$. 
\end{corollary}
Using the definition of the Moore-Penrose \ref{MPdef}, Drazin inverse \ref{DRdef}, and the Lemma \ref{rgnl}, we can show the following result under $M$-product. 
\begin{theorem}\label{thm:dfrntginv}
Let  Let $M\in\mb{C}^{p\times p}$, $\mc{A}\in\mathbb{C}^{m\times n\times p}$ and $\mc{B}\in \mathbb{C}^{m\times m\times p}$ with tubal index $k$. Then 
    \begin{enumerate}
        \item[ (i)] $\mc{A}^{\dagger}=\mc{A}^{(2)}_{\displaystyle\rg(\mc{A}^*),\nl(\mc{A}^*)}$.
        \item[ (ii)] $\mc{B}^{D}=\mc{B}^{(2)}_{\displaystyle\rg(\mc{B}^k),\nl(\mc{B}^k)}$.
    \end{enumerate}
    \end{theorem}
 The representation of an outer inverse based on a FRD is presented in the next result.
\begin{theorem}\label{eq:ats2}
    Let $M\in\mbc^{p\times p},~\mc{A}\in \mbc^{m\times n\times p}$ with $\mrk{\mc{A}}=\mathbf{r}$ and $\mc{W}\in \mbc^{n\times m\times p}$ such that $\mathscr{R}(\mc{W})=\mc{T}$ and $\mathscr{N}(\mc{W})=\mc{S}$, where $\tilde{\mc{T}}(:,1,k)$ is a subspace of $\mbc^{n}$ of dimension $l\leq r$, and $\Tilde{\mc{S}}(:,1,k)$ is a subspace of $\mbc^{m}$ of dimension $m-l$ for each $k=1,2,\ldots,p$. If $\mc{W}=\mc{S}\m\mc{T}$ is a FRD of $\mc{A}$ and $\mc{A}_{\mc{T}, \mc{S}}^{(2)}$ exists, then 
    \begin{itemize}
        \item[(i)] $\mc{T}\m\mc{A}\m\mc{S}$ is invertible.
        \item[(ii)] $\mc{A}_{\mc{T}. \mc{S}}^{(2)}=\mc{S}\m(\mc{T}\m\mc{A}\m\mc{S})^{-1}\m\mc{T}=\mc{A}^{(2)}_{\mathscr{R}(\mc{S}),\mathscr{N}(\mc{T})}$.
    \end{itemize}
\end{theorem}
\begin{proof}
(i) Let $\Tilde{\mc{A}}=\mc{A}\times_3 M$, $\Tilde{\mc{W}}=\mc{W}\times_3 M$, $\Tilde{\mc{S}}=\mc{S}\times_3 M$ and $\Tilde{\mc{T}}=\mc{T}\times_3 M$. By Corollary \ref{cor:rngnul}, we also have  $\mathscr{R}(\mc{S}) = \mc{T}$, $\mathscr{N}(\mc{T}) = \mc{S}$.  To show $\mc{T}\m\mc{A}\m\mc{S}$ is an invertible tensor, it is sufficient to show that $\Tilde{\mc{T}}(:,:,k)\Tilde{\mc{A}}(:,:,k)\Tilde{\mc{S}}(:,:,k)$ is invertible for all $k$, $k
=1,\ldots,p$. Since $\mathrm{rank}(\Tilde{\mc{A}}(:,:,k))=r$ and $\mc{W}=\mc{S}\m\mc{T}$ is FRD of $\mc{W}$, so by Proposition \ref{prop1.11} (i), we have  $\Tilde{\mc{W}}(:,:,k)=\Tilde{\mc{S}}(:,:,k)\Tilde{\mc{T}}(:,:,k)$ is a FRD for all $k$, $k
=1,\ldots,p$. Further, from $\mathscr{R}(\mc{W})=\mc{T}$ and $\nl(\mc{W})=\mc{S}$, we obtain $R(\Tilde{\mc{W}}(:,:,k))=\tilde{\mc{T}}(:,1,k)$ and $N(\Tilde{\mc{W}}(:,:,k))=\tilde{\mc{S}}(:,1,k)$. Hence by the Theorem 3.1 of \cite{sheng2007}, we get $\Tilde{\mc{T}}(:,:,k)\Tilde{\mc{A}}(:,:,k)\Tilde{\mc{S}}(:,:,k)$ is invertible for all $k$. \\    
(ii) Let $\mc{X}=\mc{S}\m(\mc{T}\m\mc{A}\m\mc{S})^{-1}\m\mc{T}$. Then 
    \begin{eqnarray*}
\mc{X}\m\mc{A}\m\mc{X}&=&\mc{S}\m(\mc{T}\m\mc{A}\m\mc{S})^{-1}\m\mc{T}A\mc{S}\m(\mc{T}\m\mc{A}\m\mc{S})^{-1}\m\mc{T}\\
&=&\mc{S}\m(\mc{T}\m\mc{A}\m\mc{S})^{-1}\m\mc{T}=\mc{X}.    
    \end{eqnarray*}
From Lemma \ref{lma:rannul} we get $\mathscr{R}(\mc{X}) = \mc{T}$, $\mathscr{N}(\mc{X}) = \mc{S}$ and hence completes the proof.  \qedhere
\end{proof}
In view of the Theorem \ref{thm:dfrntginv} and Theorem \ref{eq:ats2}, we obtain the following result as corollary. 
\begin{corollary}\label{cor:fr}
    Let $M\in\mb{C}^{p\times p}$, $\mc{A}\in\mathbb{C}^{m\times n\times p}$  with  $\mrk{\mc{A}}=\mathbf{r}$ and $\mc{B}\in \mathbb{C}^{m\times m\times p}$ with tubal index $k$. 
    \begin{enumerate}
        \item[(i)] If $\mc{A}^*=\mc{S}\m\mc{T}$ is a FRD of $\mc{A}^*$ then $\mc{A}^\dagger=\mc{S}\m(\mc{T}\m\mc{A}\m\mc{S})^{-1}\m\mc{T}$.
        \item[(ii)] If $\mc{B}^k=\mc{S}\m\mc{T}$ is a FRD of $\mc{B}^k$ then $\mc{B}^{D}=\mc{S}\m(\mc{T}\m\mc{B}\m\mc{S})^{-1}\m\mc{T}$.
    \end{enumerate}
\end{corollary}
\begin{remark}\label{rmk2.12}
In corollary \ref{cor:fr}, if we assume $\mrk{\mc{A}}=(r_{1},r_{2},\ldots, r_{p})$ with $r_i\neq r_j$ for some $i\neq j$ and multi index for the tensor $\mc{B}$ ( while computing the Moore-Penrose inverse or Drazin inverse) still the result is true as presented in Algorithm \ref{alg:outInMPIDR}. 
\end{remark}
In support of the Remark \ref{rmk2.12}, the following algorithms  is developed to compute the outer inverse (specifically for the Moore-Penrose and Drazin inverse) of $\mc{A}$, when the rank or index of the frontal slices of $\Tilde{\mc{A}}$ are not necessarily the same. 
\begin{algorithm}[H]
 \caption{Computation of $\mc{A}^\dagger=\mc{A}_{\rg(\mc{A^*}), \nl(\mc{A^*})}^{(2)}$ or $\mc{B}^D=\mc{B}_{\rg(\mc{A^k}), \nl(\mc{A^k})}^{(2)}$} \label{alg:outInMPIDR}
\begin{algorithmic}[1]
\State {\bf Input} $M\in\mb{C}^{p\times p}$, $\mc{A}\in \mb{C}^{m\times n\times p}$ or $\mc{B}\in \mb{C}^{m\times m\times p}$
\State {\bf Compute } $\Tilde{\mc{A}}=\mc{A}\times_{3}M$ or $\Tilde{\mc{B}}=\mc{B}\times_{3}M$
\For{$k=1:p$}
\State {\bf Compute } $r=\mbox{rank}(\Tilde{\mc{A}}(:,:,k)$ or $l=\ind(\Tilde{\mc{B}}(:,:,k)$
\State $[F,G]=\mbox{FRD}(\Tilde{\mc{A}}(:,:,k))$ or $(\Tilde{\mc{B}}(:,:,k))^l$
\State $\Tilde{\mc{X}}(:,:,k)=F(G AF)^{-1}G$ or $\Tilde{\mc{Y}}(:,:,k)=F(G BF)^{-1}G$
\EndFor
\State $\mc{X}=\Tilde{\mc{X}}\times_{3}M^{-1}$ or $\mc{Y}=\Tilde{\mc{Y}}\times_{3}M^{-1}$
\State \Return $\mc{A}^{\dg}=\mc{X}$ or $\mc{A}^D=\mc{Y}$
 \end{algorithmic}
\end{algorithm}
The following examples are worked out to demonstrate the above Algorithm \ref{alg:outInMPIDR}.
\begin{example}\rm\label{drzexa}
    Let $\mc{A}\in\mbc^{3\times 3\times 3}$ with frontal slices are given by 
    \[\mc{A}(:,:,1)=\begin{pmatrix}
        1   &  0 &    0\\
	     3   &  3  & -1\\
	     0   &  0 &    0
    \end{pmatrix},~\mc{A}(:,:,2)=\begin{pmatrix}
        2   &  0 &    0\\
	     0    & 0   &  1\\
	     0   &  0  &   0
    \end{pmatrix},~\mc{A}(:,:,3)=\begin{pmatrix}
        0    & 1&     0\\
	     0    & 0   &  1\\
	     0   &  0  &   0
    \end{pmatrix},\mbox{ and }M=\begin{pmatrix}
        1  &   0&     1\\
     0   &  1   &  0\\
     0   &  0  &   1
    \end{pmatrix}.\]
 We can compute $\ind(\Tilde{\mc{A}}(:,:,1))=1$,~$\ind(\Tilde{\mc{A}}(:,:,2))=2$ and $\ind(\Tilde{\mc{A}}(:,:,3))=3$. Using Algorithm \ref{alg:outInMPIDR}, we evaluate $\mc{Y}:=\mc{A}^D$, where 
 \[\mc{Y}(:,:,1)=\begin{pmatrix}
        0.0625   & 0.0625   &      0\\
	    0.1875  &  0.1875  &       0\\
	         0      &   0 &        0
    \end{pmatrix},~\mc{Y}(:,:,2)=\begin{pmatrix}
        0.5   &  0 &    0\\
	     0    & 0   &  0\\
	     0   &  0  &   0
    \end{pmatrix},~\mc{Y}(:,:,3)=\begin{pmatrix}
        0    & 0&     0\\
	     0    & 0   &  0\\
	     0   &  0  &   0
    \end{pmatrix}.\]
    Also, we find the errors 
    \[\|\mc{Y}\m\mc{A}^k-\mc{A}^k\|_{F}=3.14e^{-16},~\|\mc{Y}\m\mc{A}\m\mc{Y}-\mc{Y}\|_{F}=1.55e^{-17},~\|\mc{Y}\m\mc{A}-\mc{A}\m\mc{Y}\|_{F}=1.35e^{-16}.\]
\end{example}
\begin{example}\rm
   Let $M$ be as defined in Example \ref{drzexa} and $\mc{A}\in\mbc^{2\times 4\times 3}$ with frontal slices are given by 
   \[\mc{A}(:,:,1)=\begin{pmatrix}
       -1     &  2   &    3   &    0\\
	    -1     &  1   &   -3   &    2\\
	     0     &  2   &    2   &   -2
    \end{pmatrix},~\mc{A}(:,:,2)=\begin{pmatrix}
        1     &  1   &    1    &   0\\
	     0    &   0   &   -1    &   0\\
	     1     &  1   &    0   &    0
    \end{pmatrix},~\mc{A}(:,:,3)=\begin{pmatrix}
      2    &  -2   &   -2   &    0\\
	     0   &   -1     &  2   &   -2\\
	     2    &  -2    &   0    &   2
    \end{pmatrix}.\]
     We can find that $\mrk{\mc{A}}=\{1,2,3\}$ and by applying Algorithm \ref{alg:outInMPIDR}, we have
     $\mc{Z}:=\mc{A}^\dagger$, where 
 \[\mc{Z}(:,:,1)=\begin{pmatrix}
  -\frac{11}{516}      &  -\frac{67}{516}     &     \frac{8}{129}  \\ 
	       \frac{5}{43}        &   \frac{7}{43}      &     \frac{5}{43}  \\   
	     \frac{187}{516}    &  -\frac{151}{516}  &       -\frac{7}{129} \\  
	      \frac{19}{86}        &    \frac{9}{43}        &  -\frac{12}{43}  
    \end{pmatrix},~\mc{Z}(:,:,2)=\begin{pmatrix}
       \frac{1}{6}     &       \frac{1}{6}     &       \frac{1}{3}   \\  
	      \frac{1}{6}          &  \frac{1}{6}     &      \frac{1}{3}   \\  
	      \frac{1}{3}         &  -\frac{2}{3}     &       -\frac{1}{3}  \\   
	       0        &      0       &       0       
    \end{pmatrix},~\mc{Z}(:,:,3)=\begin{pmatrix}
         \frac{9}{86}    &       \frac{2}{43}      &     \frac{9}{86}   \\
	      -\frac{5}{43}     &     -\frac{7}{43}        &  -\frac{5}{43}    \\
	     -\frac{12}{43}        &   \frac{9}{43}    &     \frac{19}{86}    \\
	     -\frac{19}{86}        &  -\frac{9}{43}    &      \frac{12}{43} 
    \end{pmatrix}.\]
    Also, we find the errors 
    \[\|\mc{A}\m\mc{Z}\m\mc{A}-\mc{A}\|_{F}=7.8505e^{-16},~~\|\mc{Z}\m\mc{A}\m\mc{Z}-\mc{Z}\|_{F}=2.2559e^{-16},\]
\[\|\mc{A}\m\mc{Z}-(\mc{A}\m\mc{Z})^{T}\|_{F}=3.6189e^{-16},~~\|\mc{Z}\m\mc{A}-(\mc{Z}\m\mc{A})^{T}\|_{F}=4.5097e^{-16}.\]
\end{example}
\section{$M$-$\mc{QDR}$ decomposition based on $M$-product}\label{sec:tqdr}
In this section we introduce the notion of $M$-$\mc{QDR}$ decomposition for third order tensors under $M$-product structure. The computation of outer inverses and specifically the Moore-Penrose and Drazin inverse based on  $M$-$\mc{QDR}$ decomposition are discussed.  First, we define $M$-$\mc{QDR}$ decomposition as follows.
\begin{definition}\label{def:tqdr}
    Let $\mc{A}\in \mbc(x)^{m\times n\times p}$ and $\mathrm{rank}_{t}(\mc{A})=r
    $. The decomposition 
    \begin{equation}\label{eq:qdr}
    \mc{A}=\mc{Q}\m\mc{D}\m\mc{R},
    \end{equation}
    is called the $M$-$\mc{QDR}$ decomposition of $A$ if $\mc{Q}\in \mbc(x)^{m\times r\times p}$, $\mc{D}\in \mbc(x)^{r\times r\times p}$ and $\mc{R}\in \mbc(x)^{r\times n\times p}$ are satisfies
    \begin{center}
 $\Tilde{\mc{Q}}(:,j,:)\neq \mathbf{0}$, $\Tilde{\mc{R}}(i,:,:)\neq \mathbf{0}$,  $\Tilde{\mc{D}}(i,:,:)\neq \mathbf{0}$  and  $\Tilde{\mc{D}}(:,j,:)\neq \mathbf{0}$ for $i=1,2,\ldots, r$ and $j=1,2,\ldots, r$.   
    \end{center}
\end{definition}
 Following this, we present an algorithm designed to compute $M$-$\mc{QDR}$ decomposition of $\mc{A}\in \mbc(x)^{m\times n\times p}$.
\begin{algorithm}[H]
 \caption{$M$-$\mc{QDR}$ decomposition of a tensor $\mc{A}$} \label{alg:qdr}
\begin{algorithmic}[1]
\State {\bf Input} $\mc{A}\in \mbc(x)^{m\times n\times p}, M\in\mbc(x)^{p\times p}$
\State {\bf Compute }  $\Tilde{\mc{A}}=\mc{A}\times_3 M$ and $r=\mathrm{rank_{t}}(\mc{A})$
 \For{$i=1:p$}
 \State {\bf Compute } $r_1=\mathrm{rank_{t}}(\Tilde{\mc{A}}(:,:,i))$;
\State $[Q,D,R]=QD\mathrm{R}(\Tilde{\mc{A}}(:,:,i))$; \Comment{Apply Algorithm 2.1 \cite{stan12}}

\If{$s< r$} 
    \State $\Tilde{\mc{Q}}(:,:,i)=\begin{bmatrix}Q&\mathbf{0}_{m,r-r_1}\end{bmatrix}$;
    \State $\Tilde{\mc{D}}(:,:,i)=\begin{bmatrix}D&\mathbf{0}_{r_1,r-r_1}\\
    \mathbf{0}_{r-r_1,r_1}&\mathbf{0}_{r-r_1,r-r_1}\end{bmatrix}$;
    \State $\Tilde{\mc{R}}(:,:,i)=\begin{bmatrix}R\\
    \mathbf{0}_{r-r_1,n}\end{bmatrix}$; 
\Else
\State $\Tilde{\mc{Q}}(:,:,i)=Q,\Tilde{\mc{D}}(:,:,i)=D,\Tilde{\mc{R}}(:,:,i)=R$;
\EndIf
\EndFor
\State $\mc{Q}=\Tilde{\mc{Q}}\times_3 M^{-1},\mc{D}=\Tilde{\mc{D}}\times_3 M^{-1},\mc{R}=\Tilde{\mc{R}}\times_3 M^{-1}$;
\State \Return $\mc{Q},\mc{D},\mc{R}$.
 \end{algorithmic}
\end{algorithm}
One numerical example is presented below to illustrate Algorithm \ref{alg:qdr} for a tensor over $\mbc$.
\begin{example}\rm
    Let $M$ be as defined in Example \ref{drzexa} and $\mc{A}\in\mbc^{3\times 4\times 3}$ with
\[\mc{A}(:,:,1)=\begin{pmatrix}
    4   &  2  &   -2  &   -1\\
	     4   &   2   &  -4  &    0\\
	     2   &   3  &   -2   &   1
\end{pmatrix},~\mc{A}(:,:,2)=\begin{pmatrix}
   1   &  -2   &   3  &    2\\
	     2  &   -2  &    4  &    2\\
	     2   &  -2  &    4  &    2
\end{pmatrix},~\mc{A}(:,:,3)=\begin{pmatrix}
   -2   &  -2   &   0   &  -1\\
	    -2  &   -2   &   2  &   -2\\
	    -1    & -3   &   1 &    -2
\end{pmatrix}.\]
Clear $\mathrm{rank}_{t}(\mc{A})=3$ since $\mrk{\mc{A}}=\{1,2,3\}$. Applying Algorithm \ref{alg:qdr}, we compute $\mc{Q},~\mc{D}$ and $\mc{R}$ with respective frontal slices are given by 
\[\mc{Q}(:,:,1)=\begin{pmatrix}
   4         &    -4/9       &     1  \\     
	       4      &       -4/9     &      -1  \\     
	       2         &    16/9    &        0  
\end{pmatrix},~\mc{Q}(:,:,2)=\begin{pmatrix}
  1          &   -8/9      &      0   \\
	       2         &     2/9    &        0     \\  
	       2     &         2/9   &         0
\end{pmatrix},~\mc{Q}(:,:,3)=\begin{pmatrix}
  -2          &    4/9   &        -1    \\   
	      -2       &       4/9   &         1 \\      
	      -1        &    -16/9     &       0   
\end{pmatrix},\]
\[\mc{D}(:,:,1)=\begin{pmatrix}
  0      &         0      &        0     \\  
	       0         &     -9/32     &       0  \\     
	       0         &      0        &      -1/2 
\end{pmatrix},~\mc{D}(:,:,2)=\begin{pmatrix}
  1/9        &     0    &           0   &    \\
	       0          &     9/8      &       0 \\      
	       0       &        0       &        0 
\end{pmatrix},~\mc{D}(:,:,3)=\begin{pmatrix}
 1/9          &   0       &        0  \\     
	       0           &    9/32     &       0   \\    
	       0            &   0    &           1/2    
\end{pmatrix},\]
\[\mc{R}(:,:,1)=\begin{pmatrix}
0         &    -11     &         -4     &        -17   \\    
0   & -\frac{32}{9}    & \frac{8}{9}    &   -\frac{20}{9}\\    
	       0         &      0      &        -2     &          1 
\end{pmatrix},~\mc{R}(:,:,2)=\begin{pmatrix}
  9         &    -10     &         19         &     10    \\   
	       0      &\frac{8}{9}   & -\frac{8}{9}   &-\frac{8}{9} \\   
	       0  &     0     &          0      &         0 
\end{pmatrix},~\mc{R}(:,:,3)=\begin{pmatrix}
9    &          11     &         -5     &          8     \\  
0      &    \frac{32}{9}& -\frac{8}{9}  & \frac{20}{9}  \\   
0          &     0      &         2        &      -1  
\end{pmatrix}.\]
Moreover, $\mrk{\mc{Q}}=\{1,2,3\}=\mrk{\mc{R}}$ and $\|\mc{A}-\mc{Q}\m\mc{D}\m\mc{R}\|_F=8.3081e-{16}$.
\end{example}
The Computation of outer inverses using $M$-$\mc{QDR}$ decomposition are discussed in the next result.

\begin{theorem}\label{thm:main}
   Let $\mc{M}\in \mbc(x)^{p\times p}$, $\mc{A}\in \mbc(x)^{m\times n\times p}$ and $\mc{W}\in \mbc(x)^{n\times m\times p}$, $\mrk{\mc{A}}=\mathbf{r}$ and $\mrk{\mc{W}}=\mathbf{s}$ with $s\leq r$. Consider 
    $\mc{W}=\mc{Q}\m\mc{D}\m\mc{R}$ be $M$-$\mc{QDR}$ decomposition of $\mc{W}$,  where $\mc{Q}\in\mb{R}^{m\times s\times p}$, $\mc{D}\in\mb{R}^{s\times s\times p}$ and $\mc{R}\in\mb{R}^{s\times n\times p}$. If $\mrk{\mc{R}\m\mc{A}\m\mc{Q}}=\mrk{\mc{W}}$ and $\mc{A}_{\mc{T}, \mc{S}}^{(2)}$ exists, 
then
\begin{equation}\label{eqqdr}
    \mc{A}_{\rg(\mc{W}), \nl(\mc{W})}^{(2)}=\mc{Q}\m(\mc{R}\m\mc{A}\m\mc{Q})^{-1}\m\mc{R} =\mc{A}_{\rg(\mc{Q}), \nl(\mc{R})}^{(2)}.
\end{equation}
\end{theorem}
\begin{proof}
Let $\mc{W}=\mc{Q}\m\mc{D}\m\mc{R}$ be the  $M$-$\mc{QDR}$ decomposition $\mc{W}$. Then by the Algorithm \ref{alg:qdr}, we have   $\Tilde{\mc{W}}(;,:,k)=\wtilde{\mc{Q}\m\mc{D}}(:,:,k)\Tilde{\mc{R}}(;,:,k)$ is FRD of $\Tilde{\mc{W}}(;,:,k)$ for all $k$, $k=1,2,\ldots p$. Thus the invertibility of $\mc{R}\m\mc{A}\m\mc{Q}\m\mc{D}$  is obtained by Theorem \ref{eq:ats2}. From the Algorithm \ref{alg:qdr}, it is easily seen that the tensor $\mc{D}$  is invertible. Hence  $(\mc{R}\m\mc{A}\m\mc{Q} \m\mc{D})^{-1}=\mc{D}^{-1}\m(\mc{R}\m\mc{A}\m\mc{Q})^{-1}$. Using Lemma  \ref{lma:rannul}, we can conclude that $\rg(\mc{W})=\rg(\mc{Q})$. Hence by Theorem \ref{eq:ats2}.
\[\mc{A}_{\rg(\mc{W}), \nl(\mc{W})}^{(2)}=\mc{Q}\m\mc{D}\m(\mc{R}\m\mc{A}\m\mc{Q}\m\mc{D})^{-1}\m\mc{R}=\mc{Q}\m(\mc{R}\m\mc{A}\m\mc{Q})^{-1}\m\mc{R} =\mc{A}_{\rg(\mc{Q}), \nl(\mc{R})}^{(2)}.\]
\end{proof}
\begin{remark}
 In the equation \eqref{eqqdr} of the Theorem \ref{thm:main}, we can replace the term $(\mc{R}\m\mc{A}\m\mc{Q})^{-1}\m\mc{R}$ by the solution of the multilinear system $(\mc{R}\m\mc{A}\m\mc{Q})\m\mc{Z}=\mc{R}$ since the computation of inverse is more computationally expensive. 
\end{remark}
\begin{remark}
  In Theorem \ref{thm:main}, if we assume $\mrk{\mc{A}}=(r_{1},r_{2},\ldots, r_{p})$ with $r_i\neq r_j$ for some $i\neq j$ and $\mrk{\mc{W}}=(s_{1},s_{2},\ldots, s_{p})$ with $s_i\neq s_j$ for some $i\neq j$ and $s_i\leq r_i$ for all $i=1,2,\ldots p$ then we can still compute the outer inverse without explicitly knowing the tensor $\mc{Q}$, $\mc{D}$ and $\mc{R}$.
\end{remark}
In view of the above two remarks, we have designed the following algorithm for computing outer inverses of third order tensors with multirank stucture. 
\begin{algorithm}[H]
 \caption{Computation of $\mc{A}_{\rg(\mc{W}), \nl(\mc{W})}^{(2)}$ based on $M$-$\mc{QDR}$ decomposition} \label{alg:outerM}
\begin{algorithmic}[1]
\State {\bf Input} $M\in\mbc^{p\times p}$, $\mc{A} \in \mbc(x)^{m \times n\times p}$
\State {\bf Compute} $\mrk{\mc{A}}=\{r_1,r_2,\ldots,r_p\}$ 
\State Choose $\mc{W} \in \mbc(x)^{n \times m\times p}$ such that  $\mrk{\mc{W}}=\{s_1,s_2,\ldots,s_p\}$ and $s_i\leq r_i$
\State {\bf Compute } $\Tilde{\mc{A}}=\mc{A}\times_{3}M$ and $\Tilde{\mc{W}}=\mc{W}\times_{3}M$
\For{$k=1:p$}
\State $[Q,D,R]=\mbox{QDR decomposition}(\Tilde{\mc{W}}(:,:,k))$\Comment{Apply Algorithm 2.1 \cite{stan12}}
\If{$\mbox{rank}(\Tilde{\mc{R}}(:,:,k)\Tilde{\mc{A}}(:,:,k)\Tilde{\mc{Q}}(:,:,k))=\Tilde{\mc{W}}(:,:,k)$} 
\State {\bf Solve} $(\Tilde{\mc{R}}(:,:,k)\Tilde{\mc{A}}(:,:,k)\Tilde{\mc{Q}}(:,:,k))X=\Tilde{\mc{R}}(:,:,k)$   
\EndIf
\State $\Tilde{\mc{Z}}(:,:,k)=\Tilde{\mc{Q}}(:,:,k)X$
\EndFor
\State {\bf Compute} $\mc{Z}=\Tilde{\mc{Z}}\times_{3}M^{-1}$
\State \Return $\mc{A}_{\rg(\mc{W}), \nl(\mc{W})}^{(2)}=\mc{Z}$.
 \end{algorithmic}
\end{algorithm}
From corollary \ref{cor:fr} and Theorem \ref{thm:main}, we obtain the following result as a corollary. 
\begin{corollary}\label{co3.66}
     Let $M\in\mb{C}^{p\times p}$, $\mc{A}\in\mathbb{C}^{m\times n\times p}$  with  $\mrk{\mc{A}}=\mathbf{r}$ and $\mc{B}\in \mathbb{C}^{m\times m\times p}$ with tubal index $k$. 
    \begin{enumerate}
        \item[(i)] If $\mc{A}^*=\mc{Q}\m\mc{D}\m\mc{R}$ is a $M$-$\mc{QDR}$ decomposition of $\mc{A}^*$ then $\mc{A}^\dagger=\mc{Q}\m(\mc{R}\m\mc{A}\m\mc{Q})^{-1}\m\mc{R}$.
        \item[(ii)] If $\mc{B}^k=\mc{Q}\m\mc{D}\m\mc{R}$ is a $M$-$\mc{QDR}$ decomposition of $\mc{B}^k$ then $\mc{B}^{D}=\mc{Q}\m(\mc{R}\m\mc{B}\m\mc{Q})^{-1}\m\mc{R}$.
    \end{enumerate}
\end{corollary}
\begin{remark}
    The Corollary \ref{co3.66} can be applied for multirank and multi index as explained for FRD in Remark \ref{rmk2.12} and Algorithm \ref{alg:outInMPIDR}. Algorithm \ref{alg:outerM} can be directly used for computing the Moore-Penrose and Drazin inverse when the tensor has multirank or multi index.
\end{remark}

\subsection{Numerical Examples}\label{sec:ne}
The notations used to denote errors associated with different matrix and tensor equations in this section.
\begin{table}[H]
    \centering
    \caption{Errors corresponding with matrix and tensor equations}
    \vspace*{0.2cm}
    \renewcommand{\arraystretch}{1.2}
    \begin{tabular}{lll}
    \hline
     $\mc{E}^M_{1} = \|\mc{A}-\mc{A}\m\mc{X}\m\mc{A}\|_F$  & $\mc{E}^M_{2} = \|\mc{X}-\mc{X}\m\mc{A}\m\mc{X}\|_F$ & $\mc{E}^M_{3} =\|\mc{A}\m\mc{X}-(\mc{A}\m\mc{X})^T\|_F$ \\     
    $\mc{E}^M_{4} =\|\mc{X}\m\mc{A}-(\mc{X}\m\mc{A})^T\|_F$&$\mc{E}^M_{5} = \|\mc{A}\m\mc{X}-\mc{X}*\mc{A}\|_F$
       &$\mc{E}_{1^k} = \|\mc{X}*\mc{A}^{k} -\mc{A}^k\|_F$ \\      
       ${E}_{1} = \|{A}-{A}{X}{A}\|_F$  &${E}_{2} = \|{X}-{X}{A}{X}\|_F$&${E}_{3} =\|{A}{X}-({A}{X})^T\|_F$ \\
       ${E}_{4} =\|{X}{A}-({X}{A})^T\|_F$&${E}_{5} = \|{A}{X}-XA\|_F$
       &${E}_{1^k} = \|XA^{k+1} -{A}^k\|_F$ \\
     \hline
    \end{tabular}
       \label{tab:error}
\end{table}
\begin{example}\rm\label{exa3.8}
  Consider $M\in\mbc^{100\times 100}$ be any random matrix and $\mc{A}\in \mbc^{n\times n\times 100}$ where the frontal slices are chosen the default MATLAB test matrix $\mathbf{cycol}$ (gallery(`cycol',n)) of order $n$. The comparison analysis of errors  and mean CPT time (MT$^{\tiny{\mbox{ten}}}$ used for tensor computations and MT$^{\tiny{mat}}$ is used for matrix computations) for computing the Moore-Penrose inverse with keeping the same number of elements for both tensors and matrices of different sizes, are presented in Table \ref{tab:MPIerror-tenm-mat} and Figure \ref{fig:mpcomp} (a).
\end{example}
\begin{table}[H]
    \begin{center}
          \caption{Comparison of errors and computational time for computing the Moore-Penrose inverse tensors and matrices by using QDR decomposition for Example \ref{exa3.8}}
         \vspace*{0.2cm}
         \renewcommand{\arraystretch}{1.2}
    \begin{tabular}{cccccc}
    \hline
        Size of $\mc{A}$ & Size of $A$ & MT$^{\tiny{\mbox{ten}}}$  & MT$^{\tiny{mat}}$ & Error$^{\mbox{\scriptsize ten}}$ &Error$^{\mbox{\scriptsize mat}}$\\ 
           \hline
    \multirow{4}{*}{\small $200\times 200\times 100$} &   \multirow{4}{*}{\small $2000\times 2000$} &\multirow{4}{*}{2.02} &\multirow{4}{*}{15.47}& 
    $\mc{E}^{\mbox{\scriptsize M}}_1=5.95e^{-11}$ & $E_1=3.61e^{-13}$
    \\ 
& &  & & $\mc{E}^{\mbox{\scriptsize M}}_{2}=2.51e^{-14}$ & $E_{2}=5.9e^{-17}$
    \\  
   
    & &  & & $\mc{E}^{\mbox{\scriptsize M}}_{3}=8.49e^{-12}$ & $E_{3}=4.52e^{-15}$
    \\  
    & &  & & $\mc{E}^{\mbox{\scriptsize M}}_{4}=2.81e^{-12}$ & $E_{4}=5.27e^{-15}$
\\
 \hline
    \multirow{4}{*}{\small $300\times 300\times 100$} &   \multirow{4}{*}{\small $3000\times 3000$} &\multirow{4}{*}{5.14} &\multirow{4}{*}{62.71}& 
     $\mc{E}^{\mbox{\scriptsize M}}_1=6.18e^{-11}$ & $E_1=3.83e^{-13}$
    \\ 
& &  & & $\mc{E}^{\mbox{\scriptsize M}}_{2}=1.81e^{-14}$ & $E_{2}=3.99e^{-17}$
    \\  
   
    & &  & & $\mc{E}^{\mbox{\scriptsize M}}_{3}=6.49e^{-12}$ & $E_{3}=4.24e^{-15}$
    \\  
    & &  & & $\mc{E}^{\mbox{\scriptsize M}}_{4}=2.28e^{-12}$ & $E_{4}=4.71e^{-15}$
\\
\hline
\multirow{4}{*}{\small $400\times 400\times 100$} &   \multirow{4}{*}{\small $4000\times 4000$} &\multirow{4}{*}{9.86} &\multirow{4}{*}{155.32}& 
  $\mc{E}^{\mbox{\scriptsize M}}_1=5.21e^{-11}$ & $E_1=5.11e^{-13}$
    \\ 
& &  & & $\mc{E}^{\mbox{\scriptsize M}}_{2}=1.10e^{-13}$ & $E_{2}=4.01e^{-17}$
    \\  
   
    & &  & & $\mc{E}^{\mbox{\scriptsize M}}_{3}=5.51e^{-11}$ & $E_{3}=4.77e^{-15}$
    \\  
    & &  & & $\mc{E}^{\mbox{\scriptsize M}}_{4}=1.68e^{-11}$ & $E_{4}=5.35e^{-15}$
\\
\hline
    \end{tabular}
       \label{tab:MPIerror-tenm-mat}  
    \end{center}
\end{table}

\begin{example}\rm\label{exa3.9}
  Consider $M\in\mbc^{100\times 100}$ be any random matrix and $\mc{A}\in \mbc^{n\times n\times 100}$ where the frontal slices are chosen the default matlab test matrix $\mathbf{chow}$ (gallery(`chow',n)) of order $n$. The similar comparison analysis like example \ref{exa3.8}, we have presented in Table \ref{tab:MPIerror-tenm-mat-chow} and Figure \ref{fig:mpcomp} (b).
\end{example}
\begin{table}[H]
    \begin{center}
          \caption{Comparison of errors and computational time for computing the Moore-Penrose inverse tensors and matrices by using QDR decomposition for Example \ref{exa3.9}}
         \vspace*{0.2cm}
         \renewcommand{\arraystretch}{1.2}
    \begin{tabular}{cccccc}
    \hline
        Size of $\mc{A}$ & Size of $A$ & MT$^{\tiny{\mbox{ten}}}$  & MT$^{\tiny{mat}}$ & Error$^{\mbox{\scriptsize ten}}$ &Error$^{\mbox{\scriptsize mat}}$\\ 
 \hline
    \multirow{4}{*}{\small $300\times 300\times 100$} &   \multirow{4}{*}{\small $3000\times 3000$} &\multirow{4}{*}{28.10} &\multirow{4}{*}{414.29}& 
     $\mc{E}^{\mbox{\scriptsize M}}_1=2.42e^{-08}$ & $E_1=6.45e^{-09}$
    \\ 
& &  & & $\mc{E}^{\mbox{\scriptsize M}}_{2}=1.25e^{-10}$ & $E_{2}=4.52e^{-10}$
    \\  
   
    & &  & & $\mc{E}^{\mbox{\scriptsize M}}_{3}=3.909e^{-10}$ & $E_{3}=3.34e^{-10}$
    \\  
    & &  & & $\mc{E}^{\mbox{\scriptsize M}}_{4}=6.41e^{-10}$ & $E_{4}=4.05e^{-10}$
\\
\hline
\multirow{4}{*}{\small $400\times 400\times 100$} &   \multirow{4}{*}{\small $4000\times 4000$} &\multirow{4}{*}{66.98} &\multirow{4}{*}{1019.62}& 
  $\mc{E}^{\mbox{\scriptsize M}}_1=6.79e^{-07}$ & $E_1=3.13e^{-08}$
    \\ 
& &  & & $\mc{E}^{\mbox{\scriptsize M}}_{2}=4.60e^{-11}$ & $E_{2}=6.50e^{-10}$
    \\  
   
    & &  & & $\mc{E}^{\mbox{\scriptsize M}}_{3}=1.54e^{-10}$ & $E_{3}=4.87e^{-10}$
    \\  
    & &  & & $\mc{E}^{\mbox{\scriptsize M}}_{4}=2.25e^{-10}$ & $E_{4}=6.94e^{-10}$
\\
\hline
    \end{tabular}
       \label{tab:MPIerror-tenm-mat-chow}  
    \end{center}
\end{table}

\begin{figure}[h!]
\begin{center}
\subfigure[Slices of tensors and matrices chosen as matlab \hspace*{0.65cm}test matrix $\mathbf{cycol}$]
{\includegraphics[scale=0.42]{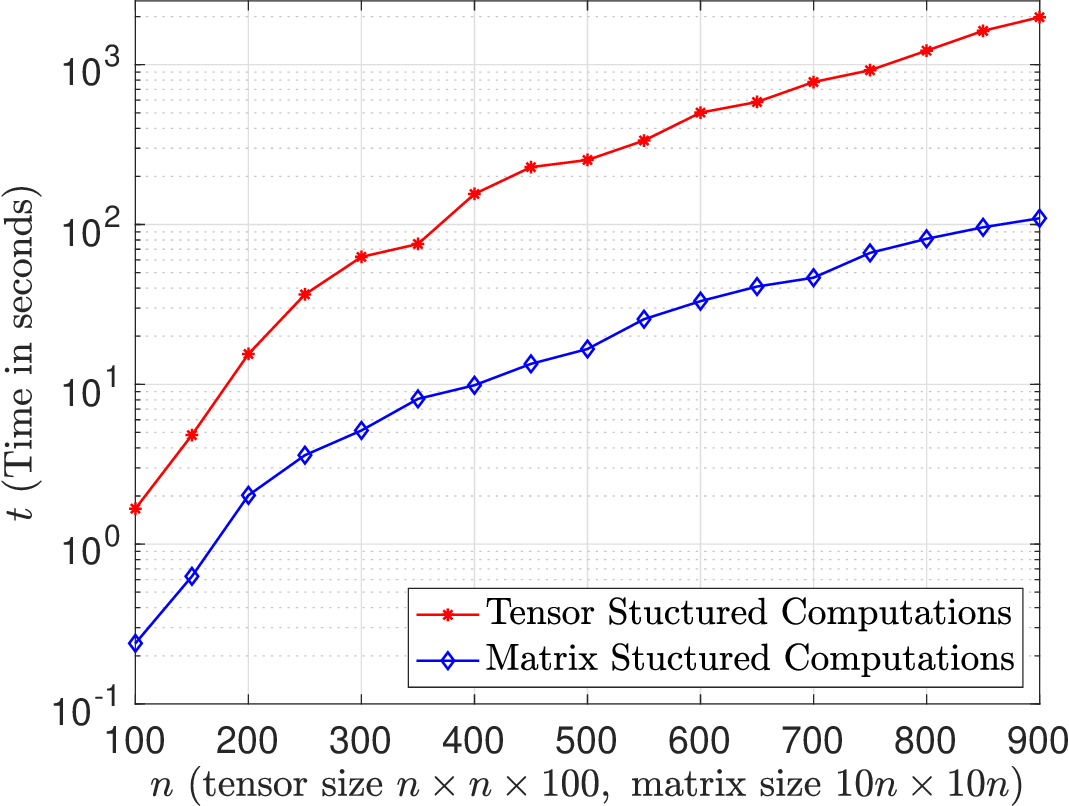}}
~~~~~\subfigure[Slices of tensors and matrices chosen as matlab \hspace*{0.65cm}test matrix $\mathbf{chow}$]{\includegraphics[scale=0.42]{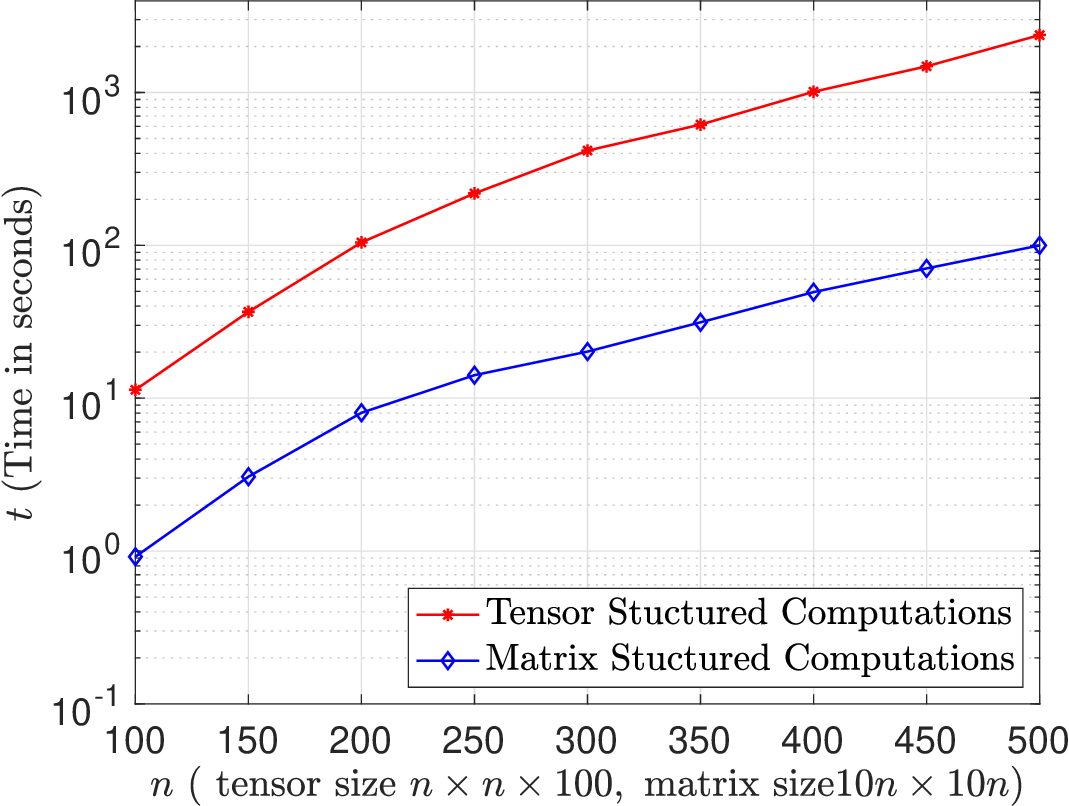}}
\caption{Comparison of mean CPU time for computing the Moore-Penrose inverse of tensors and matrices}
\label{fig:mpcomp}
\vspace{0.5cm}
\subfigure[Slices of tensors and matrices chosen as matlab \hspace*{0.65cm}test matrix $\mathbf{cycol}$]
{\includegraphics[scale=0.42]{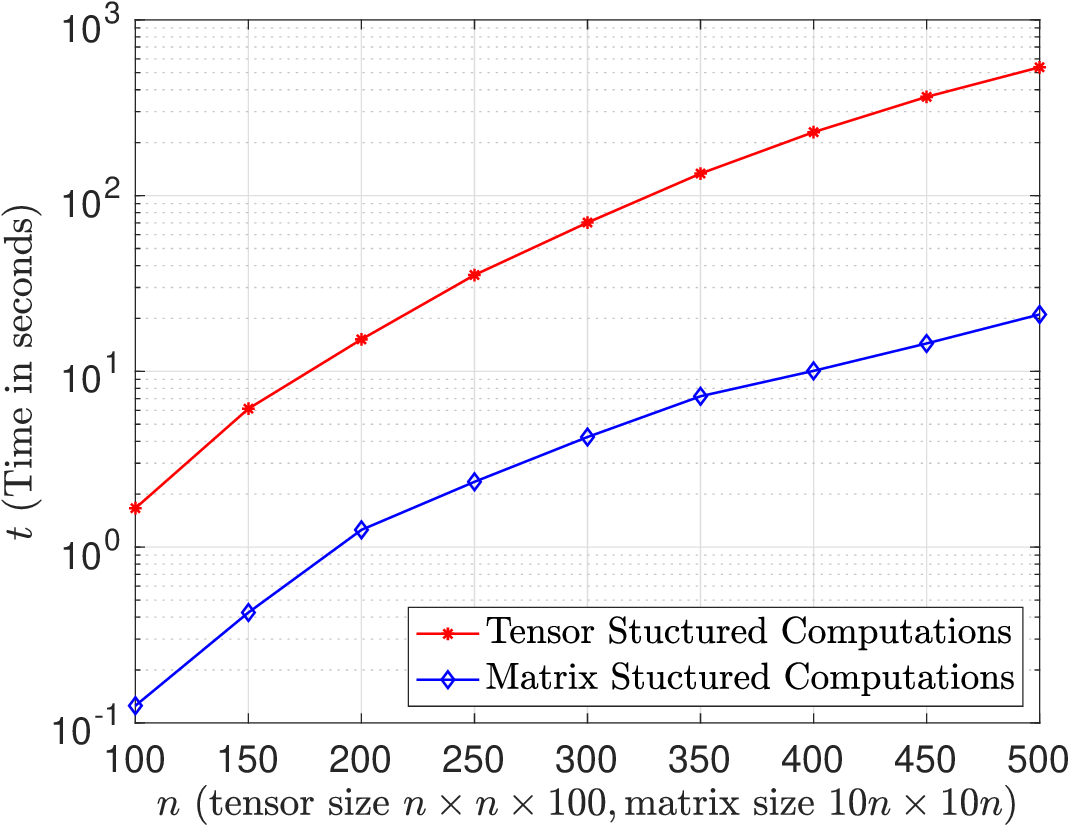}}
~~~~~\subfigure[Slices of tensors and matrices chosen as matlab \hspace*{0.65cm}test matrix $\mathbf{gearmat}$]{\includegraphics[scale=0.42]{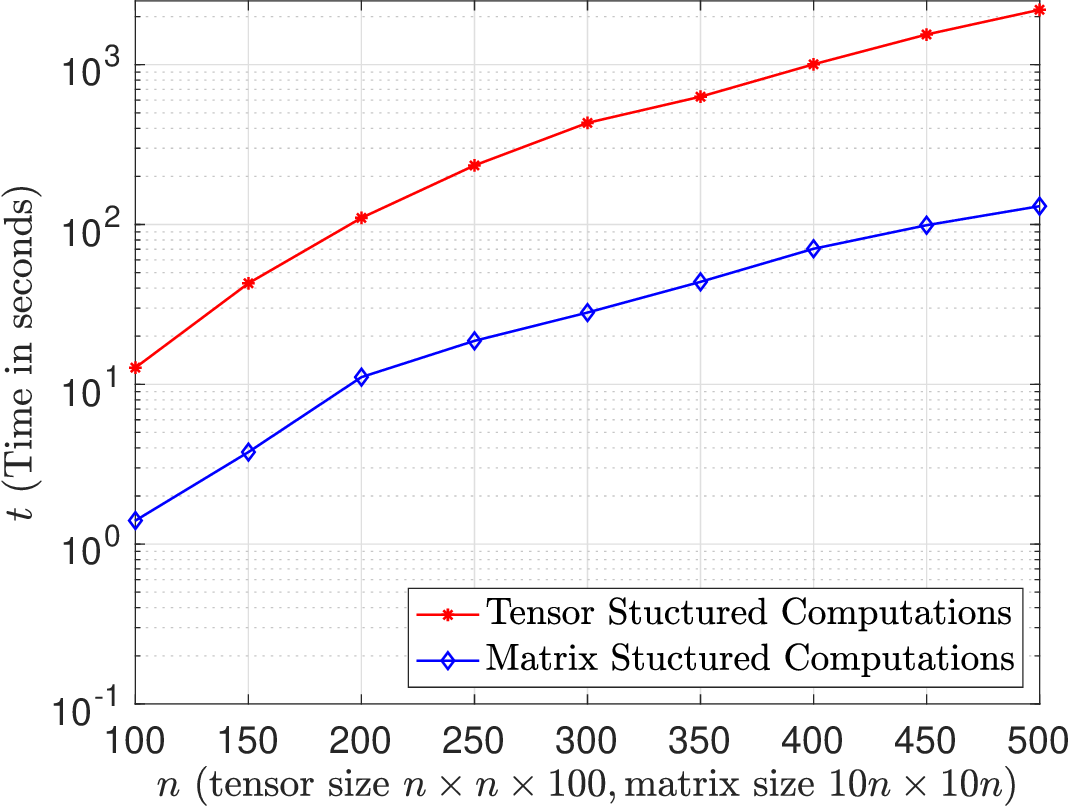}}
\caption{Comparison of mean CPU time for computing the Drazin inverse of tensors and matrices}
\label{fig:drazcomp}
\end{center}
\end{figure}
Note that the presented errors (Error$^{\mbox{\scriptsize ten}}$ and Error$^{\mbox{\scriptsize mat}}$) and mean CPU times (MT$^{\mbox{\scriptsize ten}}$  and  MT$^{\mbox{\scriptsize mat}}$) for different sizes of $\mc{A}$ and $A$ verifies the effectiveness of the algorithm in achieving accurate and efficient computations. The consistency in errors across different tensor and matrix dimensions demonstrates the robustness and reliability of the proposed approach.

\subsection{Lossy Image Compression}
Image compression is essential in our increasingly data-driven world. The direct relationship between compressed image size and storage costs is significant. Further,  compressed images require less bandwidth for transmission, which has several benefits, including reduced network congestion, lower energy consumption, and faster loading times for web pages and applications. Thus, the goal of compression is to minimize the number of bits needed to represent an image while maintaining acceptable quality. The challenge is finding the optimal balance between compression ratio and image quality. This subsection discusses lossy image compression based on $\mc{QDR}$ decomposition. This decomposition represents an interesting approach to lossy compression that balances spatial and frequency domain considerations. Applications of lossy image compression are important in specific domains where its adaptive nature is particularly beneficial.

\begin{figure}[h!]
\begin{center}
\subfigure[fruit and vegetable]{\includegraphics[height=4.5cm]{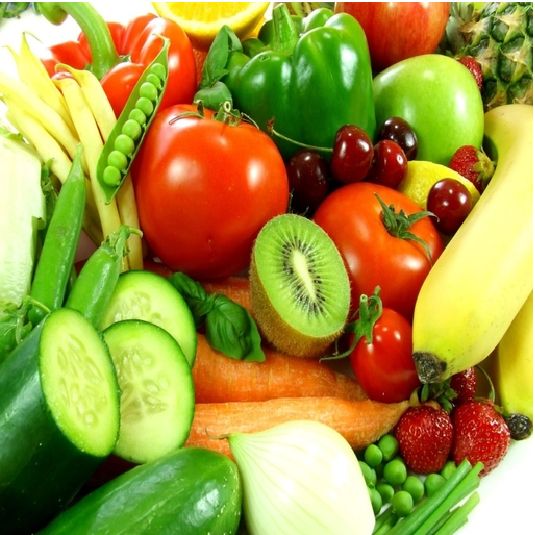}}\hspace{2cm}
\subfigure[baboon]{\includegraphics[height=4.5cm]{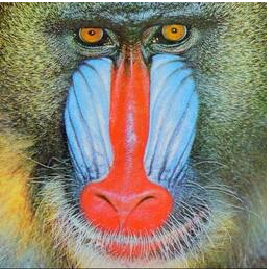}}
\caption{Original images of size: (a)  $512 \times 512 \times 3$;   (b)  $256 \times 256 \times 3$ }\label{LossyImageoriginal}
\end{center}
\end{figure}

\begin{figure}[h!]
\begin{center}
\subfigure[$k=50$]{\includegraphics[height=3cm]{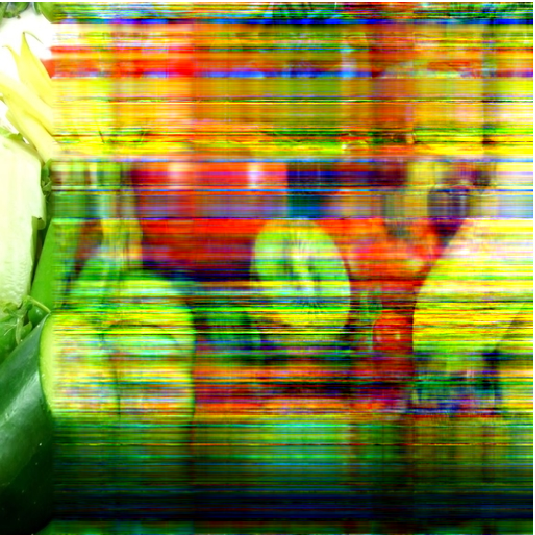}}
\subfigure[$k=100$]{\includegraphics[height=3cm]{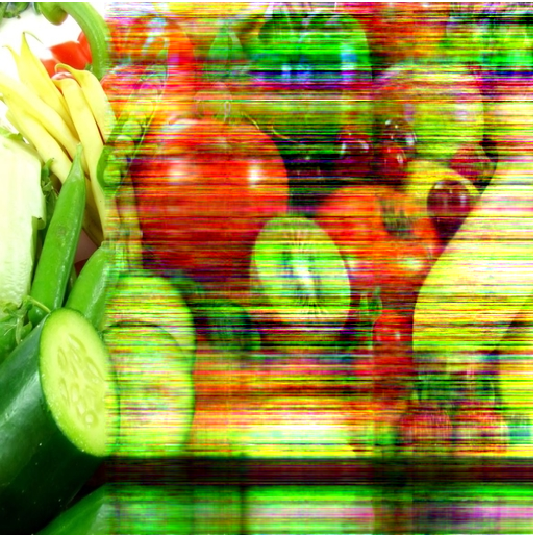}}
\subfigure[$k=150$]{\includegraphics[height=3cm]{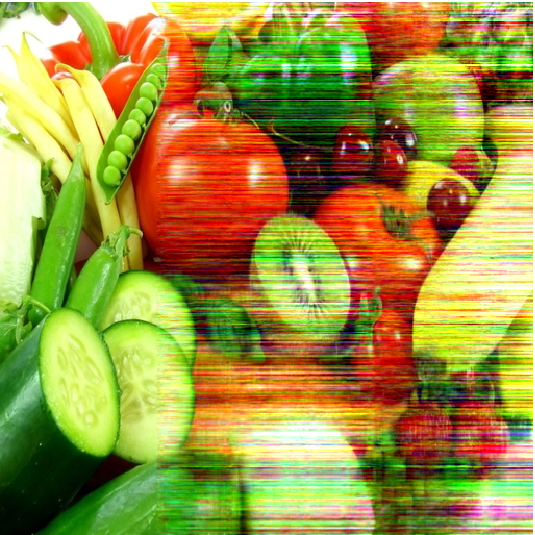}}
\subfigure[$k=200$]{\includegraphics[height=3cm]{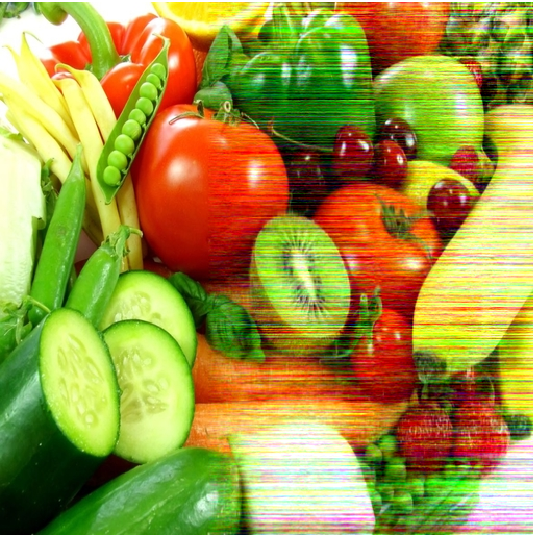}}
\subfigure[$k=250$]{\includegraphics[height=3cm]{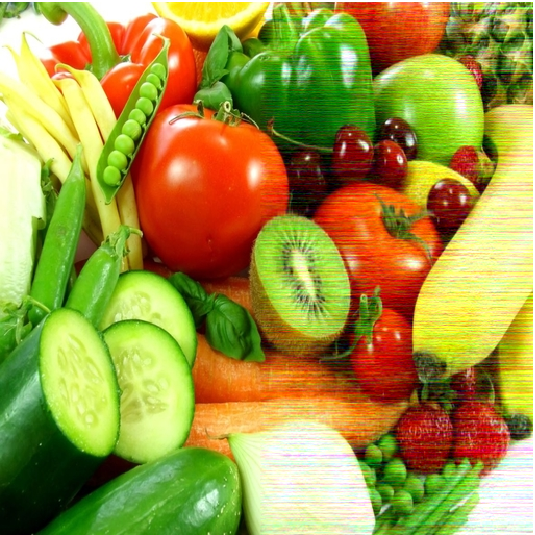}}\\
\subfigure[$k=300$]{\includegraphics[height=3cm]{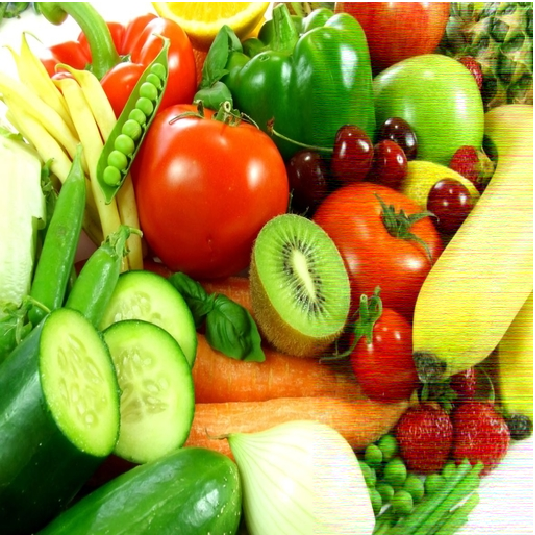}}
\subfigure[$k=350$]{\includegraphics[height=3cm]{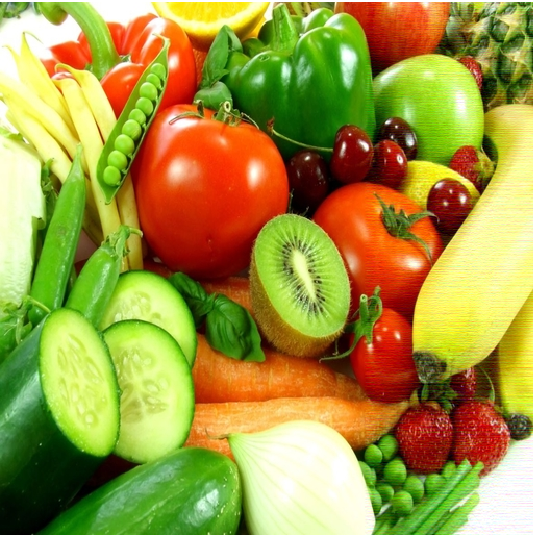}}
\subfigure[$k=400$]{\includegraphics[height=3cm]{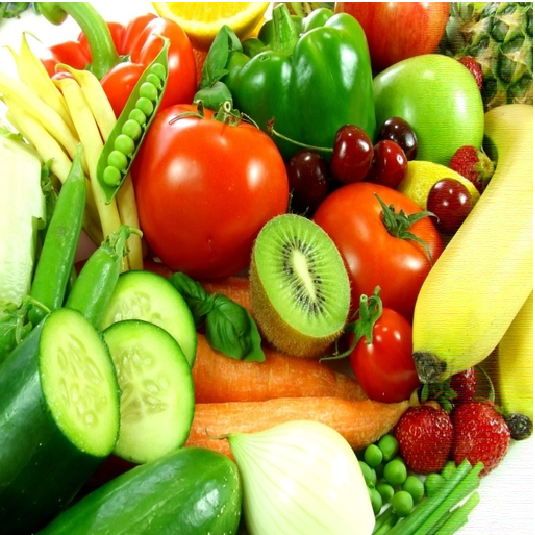}}
\subfigure[$k=450$]{\includegraphics[height=3cm]{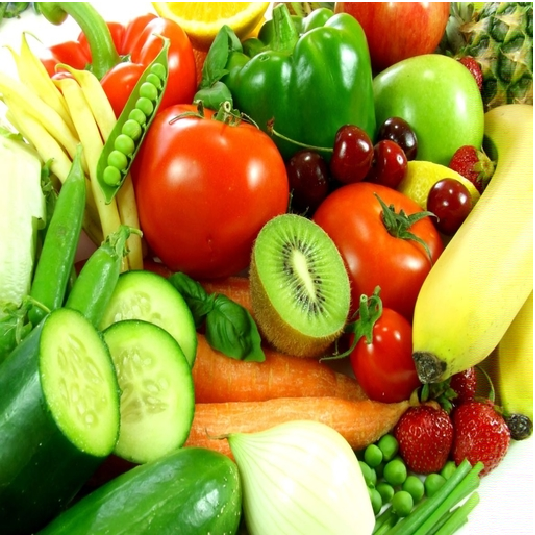}}
\subfigure[$k=500$]{\includegraphics[height=3cm]{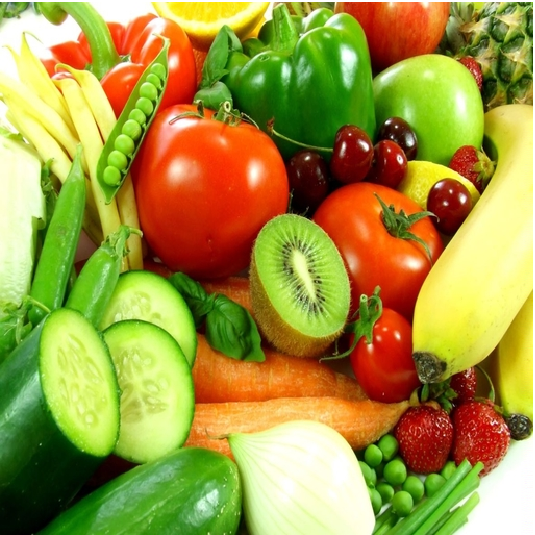}}
\caption{Compressed images of Figure \ref{LossyImageoriginal} (a) with different value of $k$}\label{compLossyImage512}
\end{center}
\end{figure}
QDR decomposition of tensor often provides an attractive algebraic properties, i.e., a good balance between compression ratio and image quality. The color lossy image compression of a tensor $\mc{A} \in \mathbb{R}^{n\times n \times 3}$ can be written in the form of QDR decomposition. 
\begin{equation*}
    \mc{A}=\mc{Q}\m\mc{D}\m\mc{R}, 
\end{equation*}
where $\mc{D} \in \mathbb{R}^{n\times n \times 3}$ is the diagonal tensor, and $\mc{R}\in \mathbb{R}^{n\times n \times 3}$ is the upper triangular tensor.  Here, the diagonal elements of $\mc{D}(:,:,i)$ decay rapidly, where $i=1,2,3$. Further, $d_{11} \geq d_{22} \geq \cdots \geq d_{kk} \geq \cdots \geq d_{nn}.$ QDR approximates this structure efficiently, concentrating most of the image information in the first $k$ columns or rows. The choice of $k$ determines the trade-off, i.e., larger $k$ represents better quality of image and lower compression, whereas smaller $k$ gives lower image quality with higher compression.
\begin{figure}
\begin{center}
\subfigure[$k=25$]{\includegraphics[height=3cm]{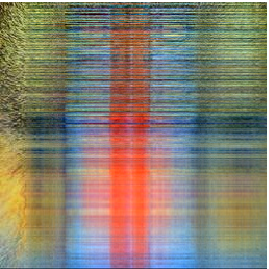}}
\subfigure[$k=50$]{\includegraphics[height=3cm]{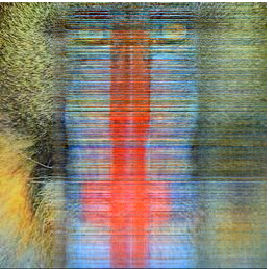}}
\subfigure[$k=75$]{\includegraphics[height=3cm]{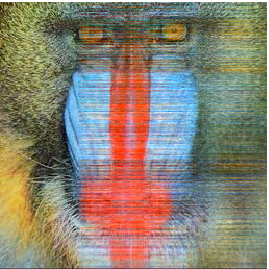}}
\subfigure[$k=100$]{\includegraphics[height=3cm]{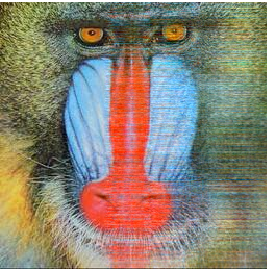}}
\subfigure[$k=125$]{\includegraphics[height=3cm]{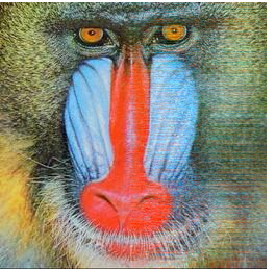}}\\
\subfigure[$k=150$]{\includegraphics[height=3cm]{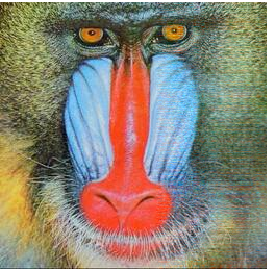}}
\subfigure[$k=175$]{\includegraphics[height=3cm]{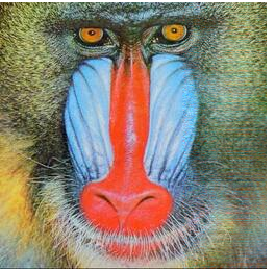}}
\subfigure[$k=200$]{\includegraphics[height=3cm]{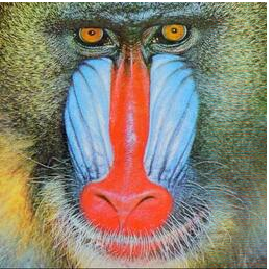}}
\subfigure[$k=225$]{\includegraphics[height=3cm]{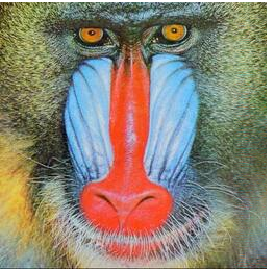}}
\subfigure[$k=250$]{\includegraphics[height=3cm]{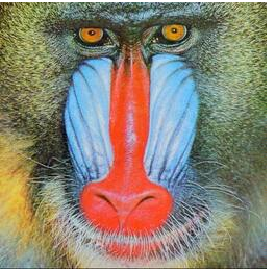}}
\caption{Compressed images of Figure \ref{LossyImageoriginal} (b) with different value of $k$}\label{LossyImage256}
\end{center}
\end{figure}

\begin{table}[H]
\centering
\caption{Comparison of PSNR  and SSIM values for different choices of $k$}
 \vspace*{0.2cm}
\begin{tabular}{|p{1cm}p{2.2cm}p{2.2cm}|p{1cm}p{2.2cm}p{2.2cm}|}
\hline
\multicolumn{3}{|c|}{Figure 4} & \multicolumn{3}{|c|}{Figure 5}  \\ \hline
\multirow{ 1}{*}{$k$} & PSNR Value & SSIM Value &\multirow{ 1}{*}{$k$}& PSNR Value & SSIM Value\\
\hline
50 & ~~~10.15 & ~~~0.6207 &25 & ~~~14.87 & ~~~0.4089 \\
100 & ~~~12.84 & ~~~0.7124 &50 & ~~~16.57 & ~~~0.5567 \\
150 & ~~~17.21 & ~~~0.8213 &75 & ~~~20.68 & ~~~0.7573 \\
200 & ~~~21.65 & ~~~0.8979 &100 & ~~~24.07 & ~~~0.8512\\
150 & ~~~26.84 & ~~~0.9514 &125 & ~~~26.65 & ~~~0.8972\\
300 & ~~~30.27 & ~~~0.9763 &150 & ~~~28.62 & ~~~0.9250 \\
350 & ~~~33.65 & ~~~0.9885 &175 & ~~~30.62 & ~~~0.9445 \\
400 & ~~~37.95 & ~~~0.9944 &200 & ~~~34.56 & ~~~0.9720 \\
450 & ~~~45.08 & ~~~0.9977 &225 & ~~~40.42 & ~~~0.9908 \\
500 & ~~~61.62 & ~~~0.9999 &250 & ~~~53.85 & ~~~0.9994 \\
\hline
\end{tabular}
\label{table2}
\end{table}
We present peak signal-to-noise ratio (PSNR) values in the Table \ref{table2}  to assess image quality quantitatively for Figure \ref{LossyImageoriginal}. This calculation effectively compares the size of the error (represented by MSE) relative to the peak value of the signal {\it{max}}. Specifically, for color image of size $n\times n \times 3$ is computed using the following 
\begin{equation*}
\text{PSNR}=\displaystyle 10 \times \log _{10}\left(\frac{max^{2}}{MSE}\right) 
\end{equation*}
Here {\it{max}} is the maximum possible pixel value of the image (e.g., 255 for 8-bit images). Further, 
${\it MSE}$ is the Mean Squared Error between the original and the processed images. See \cite{behera2023computation} for more detail. It is worth mentioning that the higher PSNR values correlate with superior image quality. PSNR is typically expressed in decibel units (dB). The values between 30-50 dB generally indicate acceptable to good image quality. However, PSNR exceeding 40 dB is considered to signify high-quality images where distortion is minimally perceptible to the human eye. In addition to this we also present another efficient tool for measuring image quality assessment for Figure \ref{LossyImageoriginal} in the Table \ref{table2} i.e., structural similarity index (SSIM). SSIM is a perceptual metric that quantifies image quality degradation caused by processing such as data compression or by losses in data transmission. SSIM values range from $-1$ to $1$, where $1$ indicates perfect similarity between the original and processed images. See \cite{behera2023computation} for computational detail. 

\section{Computation of outer inverses via symbolic computation}\label{sec:outsy}
\begin{example}\rm
 Let $\mc{A}\in\mbc(x)^{2\times 2\times 2}$  with frontal slices are given by 
 \[\mc{A}(:,:,1)=\begin{pmatrix}
     1+x &1+x\\
     2+x & 1-x
 \end{pmatrix},\mbox{ and }\mc{A}(:,:,2)=\begin{pmatrix}
     x &x\\
     x & x
 \end{pmatrix}.\]
Applying algorithm \ref{alg:qdr}, we can compute the $M$-$\mc{QDR}$ decomposition of tensor $\mc{A}$,  where  $\mc{Q}\in \mbc(x)^{2\times 2\times 2},~\mc{D}\in \mbc(x)^{2\times 2\times 2}$ and $\mc{R}\in \mbc(x)^{2\times 2\times 2}$  with respective frontal slices are given by 
\[\mc{Q}(:,:,1)=\displaystyle\begin{pmatrix}
    1+x~~&\displaystyle\frac{36x^3+72x^2+47x+10}{40x^2+60x+25}\\
    &\\
    2+x~~&\displaystyle-\frac{28x^3+46x^2+26x+5}{40x^2+60x+25}
\end{pmatrix},~\mc{Q}(:,:,2)=\displaystyle\begin{pmatrix}
x~~&\displaystyle\frac{4 x^3+8 x^2+3 x}{40 x^2+60 x+25}\\
&\\
x~~&\displaystyle-\frac{12 x^3+14 x^2+4 x}{40 x^2+60 x+25}
\end{pmatrix},\]
\[\mc{D}(:,:,1)=\begin{pmatrix}
    \frac{4 x^2+6 x+5}{40 x^2+60 x+25}&0\\
    &\\
    0&-\frac{14 x^2+16 x+5}{16 x^4+32 x^3+24 x^2+8 x+1}
\end{pmatrix},~\mc{D}(:,:,2)=\begin{pmatrix}
-\frac{4 x^2+6 x}{40 x^2+60 x+25}&0\\
&\\
0&\frac{6 x^2+4 x}{16 x^4+32 x^3+24 x^2+8 x+1}
\end{pmatrix},\]
\[\mc{R}(:,:,1)=\begin{pmatrix}
4x^2+6x+5&2x^2+x+3\\
&\\
0&\frac{56 x^4+120 x^3+98 x^2+36 x+5}{40 x^2+60 x+25}
\end{pmatrix},~\mc{R}(:,:,2)=\begin{pmatrix}
4x^2+6x&2x^2+5\\
&\\
0&\frac{24 x^4+40 x^3+22 x^2+4x}{40 x^2+60 x+25}
\end{pmatrix}.\]
\end{example}
Next, we will discuss the computation of the outer inverse of tensors whose entries are either polynomial in single variable or rational functions, based on symbolic $M$-$\mc{QDR}$ decomposition. Consider $\mc{A} \in \mbc(x)^{m \times n\times p}$ with elements are 
\begin{equation*}
    \mc{A}(i_{1},i_{2},i_{3})=\dfrac{\displaystyle\sum_{d=0}^{\overline{d}_{\mc{A}}} \overline{a}_{d,i_{1},i_{2},i_{3}}x^{d}}{\displaystyle\sum_{d=0}^{\underline{d}_{\mc{A}}} \underline{a}_{d,i_{1},i_{2},i_{3}} x^{d}},
\end{equation*}
where $\overline{d}_{\mc{A}}$ is the maximal degrees of the numerators and 
and $\underline{d}_{\mc{A}}$ is the  maximal degrees denominators of  $\mc{A}$. Let the entries of  $ \Tilde{\mc{Q}}$ and $\Tilde{\mc{R}}$ are in the following form
\begin{eqnarray}
 \Tilde{\mc{Q}}(i_{1},i_{2},i_{3})&=&\frac{\sum_{d=0}^{\overline{d}_{\mc{Q}}} \overline{q}_{d,i_{1},i_{2},i_{3}} x^{d}}{\sum_{d=0}^{\underline{d}_{\mc{Q}}} \underline{q}_{d,i_{1},i_{2},i_{3}} x^{d}}, \quad 1 \leq i_{1} \leq m, 1 \leq i_{2} \leq s, 1\leq i_{3} \leq p,\label{eq:qij}\\
 \Tilde{\mc{R}}(i_{1},i_{2},i_{3})&=&\frac{\sum_{d=0}^{\overline{d}_{\mc{R}}}\overline{r}_{d,i_{1},i_{2},i_{3}} x^{d}}{\sum_{d=0}^{\underline{d}_{\mc{R}}} \underline{r}_{d,i_{1},i_{2},i_{3}} x^{d}}, \quad 1 \leq i_{1} \leq s, i_{1} \leq i_{2} \leq n, 1\leq i_{3} \leq p,\label{eq:rij}
\end{eqnarray}
Variables with an overbar represent the coefficients of the numerator, while variables with an underline represent the coefficients of the denominator, which can be followed subsequently. The representation  of outer inverse of tensors over the field of rational functions is presented in the below theorem.
\begin{theorem}\label{thm:ats2}
Let $M\in\mbc^{p\times p}$, $\mc{A} \in \mbc(x)^{m \times n\times p}$ and  $\mc{W} \in \mbc(x)^{n \times m\times p}$ where  $\mrk{\mc{A}}={\mathbf{r}}$, $\mrk{\mc{W}}={\mathbf{s}}$ with  $s\leq r$. Let $\mc{W}=\mc{Q}\m\mc{D}\m\mc{R}$ be a $M$-$\mc{QDR}$ decomposition of  $\mc{W}$, where the entries  of $\Tilde{\mc{Q}}$ and $\Tilde{\mc{R}}$ respectively are of the form \eqref{eq:qij} and \eqref{eq:rij}. Consider $\mrk{\mc{W}}=\mrk{\mc{R}\m\mc{A}\m\mc{Q}}={\mathbf{s}}$, and denote the arbitrary $(i_{1},i_{2},i_{3})^{th}$  element of $\mc{N}=(\mc{R}*\mc{A}*\mc{Q})^{-1}$ by
\begin{equation}\label{eq:nij}
    \Tilde{\mc{N}}(i_{1},i_{2},i_{3})=\frac{\sum_{d=0}^{\overline{\eta}} \overline{n}_{d, i_{1},i_{2},i_{3}} x^{d}}{\sum_{d=0}^{\underline{\eta}} \underline{n}_{d, i_{1},i_{2},i_{3}} x^{d}}.
\end{equation}
Then the arbitrary $(i_{1},i_{2},i_{3})^{th}$ entries of $\mc{A}_{\rg(\mc{Q}), \nl(\mc{R})}^{(2)}$ is given by 
\[\left(\mc{A}_{\rg(\mc{Q}), \nl(\mc{R})}^{(2)}\right)(i_{1},i_{2},i_{3})=\frac{\overline{P}(i_{1},i_{2},i_{3})}{\underline{P}(i_{1},i_{2},i_{3})},\]
where
\begin{equation}\label{eq:gmij}
 \overline{P}(i_{1},i_{2},i_{3})=\sum_{d=0}^{\underline{P}_{q}-\underline{\alpha}_{q}+\overline{\alpha}_{q}}\left(\sum_{k=1}^{\min\{i_{2}, s\}} \sum_{\ell=1}^s \sum_{t_{1}=0}^{d} \overline{\alpha}_{d,t, k, u, i_{1},i_{2},i_{3},l, m}P_{d-t_{1}, i_{1},i_{2},i_{3}}\right) x^{d}, 
\end{equation}
\[\underline{P}(i_{1},i_{2},i_{3})=\mathrm{PolynomialLCM}\left\{\sum_{k=0}^{\underline{\alpha}_{q}} \underline{\alpha}_{t, k, u, i_{1},i_{2},i_{3},\ell, m} x^{k} \mid m=\overline{1, \min \{i_{2}, s\}}, \ell=\overline{1, s}\right\}=\sum_{d=0}^{\underline{P}_{q}} \underline{P}_{ d, i_{1},i_{2},i_{3}} x^{d},\]
with $
\overline{\alpha}_{q}=\overline{\beta}+\overline{\eta}+\overline{\delta}$, $\underline{\alpha}_{q}=\underline{\beta}+\underline{\eta}+\underline{\delta}$, $\overline{\alpha}_{t, k, u, i_{1},i_{2},i_{3},l,m}\sum_{t=0}^{k} \left(\sum_{u=0}^{k-t} \overline{q}_{t, i_{1},\ell,i_{3}} \overline{n}_{k-t-u, \ell,i_{2},i_{3}}\right)\overline{r}_{u,m,i_{2},i_{3}}$, 
\[  \underline{\alpha}_{t, k, u, i_{1},i_{2},i_{3},l,m}=\sum_{t=0}^{k} \left(\sum_{u=0}^{k-t} \underline{q}_{t, i_{1},\ell,i_{3}} \underline{n}_{k-t-u, \ell,i_{2},i_{3}}\right)\underline{r}_{u,m,i_{2},i_{3}},\]
%and for $0 \leq d \leq \underline{P}_{q}-\underline{\alpha}_{q},$ $P_{d, i_{1}, i_{2}, i_{3}, \ell, m}$ are the coefficients of the polynomial
\[P(t, k, u, i_{1},i_{2},i_{3},l, m)=\frac{\underline{P}(i_{1},i_{2},i_{3})}{ \sum_{k=0}^{\underline{\alpha}_{q}} \underline{\alpha}_{t, k, u, i_{1},i_{2},i_{3},l, m} x^{k}},\quad m=\overline{1, \min \{i_{2}, s\}},  \ell=\overline{1, s}.\]
\end{theorem}
\begin{proof}
From the given assumption in equation \eqref{eq:nij},  we obtain
\[\Tilde{\mc{N}}(:,:,i_{3})=(\Tilde{\mc{R}}(:,:,i_{3})\cdot \Tilde{\mc{A}}(:,:,i_{3})\cdot \Tilde{\mc{Q}}(:,:,i_{3}))^{-1}=\left(\Tilde{\mc{N}}(i_{1},i_{2},i_{3})\right),\] where $1\leq i_{1}, i_{2}\leq s, 1\leq i_{3}\leq p$. Thus,
\begin{eqnarray*}
\left(\Tilde{\mc{Q}}(:,:,i_{3})\cdot \Tilde{\mc{N}}(:,:,i_{3})\right)(i_{1},i_{2},i_{3})&=&\sum_{\ell=1}^{s}\Tilde{\mc{Q}}(i_{1},\ell,i_{3})\Tilde{\mc{N}}(\ell,i_{2},i_{3})\\
&=&\sum_{\ell=1}^{s}\dfrac{\sum_{k=0}^{\overline{\beta}}\overline{q}_{k,i_{1},\ell,i_{3}}x^{k}\sum_{k=0}^{\overline{\eta}}\overline{n}_{k,\ell,i_{2},i_{3}}x^{k}}{\sum_{k=0}^{\underline{\beta}}\underline{q}_{k,i_{1},\ell,i_{3}}x^{k}\sum_{k=0}^{\underline{\eta}}\underline{n}_{k,\ell,i_{2},i_{3}}x^{k}}\\
&=&\sum_{\ell=1}^{s} \frac{\sum_{k=0}^{\overline{\beta}+\overline{\eta}}\left(\sum_{j=0}^k \overline{q}_{j, i_{1},\ell,i_{3}} \overline{n}_{k-j, \ell,i_{2},i_{3}}\right) x^{k}}{\sum_{k=0}^{\underline{\beta}+\underline{\eta}}\left(\sum_{j=0}^k \underline{q}_{j, i_{1},\ell,i_{3}} \underline{n}_{k-j, \ell,i_{2},i_{3}}\right) x^{k}},
\end{eqnarray*}
and 
\begin{eqnarray*}
&&\left(\Tilde{\mc{Q}}(:,:,i_{3})\cdot\Tilde{\mc{N}}(:,:,i_{3})\cdot\Tilde{\mc{R}}(:,:,i_{3})\right)(i_{1},i_{2},i_{3})\\
&&\hspace*{2cm}=\sum_{m=1}^{\min\{i_{2},s\}}(\Tilde{\mc{Q}}(:,:,i_{3})\cdot\Tilde{\mc{N}}(:,:,i_{3}))(i_{1},m,i_{3})\Tilde{\mc{R}}(m,i_{2},i_{3})\\
&&\hspace*{2cm}=\sum_{m=1}^{\min\{i_{2},s\}}\left( \sum_{\ell=1}^{s} \frac{\sum_{k=0}^{\overline{\beta}+\overline{\eta}}\left(\sum_{j=0}^k \overline{q}_{j, i_{1},\ell,i_{3}} \overline{n}_{k-j, \ell,m,i_{3}}\right) x^{k}}{\sum_{k=0}^{\underline{\beta}+\underline{\eta}}\left(\sum_{j=0}^k \underline{q}_{j, i_{1},\ell,i_{3}} \underline{n}_{k-j, \ell,m,i_{3}}\right) x^{k}}\cdot \frac{\sum_{k=0}^{\overline{\delta}}\overline{r}_{k,m,i_{2},i_{3}}x^{k}}{\sum_{k=0}^{\underline{\delta}}\underline{r}_{k,m,i_{2},i_{3}}x^{k}}\right)\\
&&\hspace*{2cm}=\sum_{m=1}^{\min\{i_{2},s\}}\sum_{\ell=1}^s \frac{\sum_{k=0}^{\overline{\beta}+\overline{\eta}+\overline{\delta}}\left(\sum_{t=0}^{k} \left(\sum_{u=0}^{k-t} \overline{q}_{t, i_{1},\ell,i_{3}} \overline{n}_{k-t-u, \ell,i_{2},i_{3}}\right)\overline{r}_{u,m,i_{2},i_{3}}\right) x^{k}}{\sum_{k=0}^{\underline{\beta}+\underline{\eta}+\underline{\delta}}\left(\sum_{t=0}^{k} \left(\sum_{u=0}^{k-t} \underline{q}_{t, i_{1},\ell,i_{3}} \underline{n}_{k-t-u, \ell,i_{2},i_{3}}\right)\underline{r}_{u,m,i_{2},i_{3}}\right) x^{k}}.
\end{eqnarray*}
Applying equation \eqref{eqqdr}, we obtain  the $(i_{1}, i_{2},i_{3})^{th}$ element of $\Tilde{\mc{A}}_{\rg(\mc{Q}), \nl(\mc{R})}^{(2)}$ as follows:
\begin{equation}\label{eq:elmentats2}
    \left(\Tilde{\mc{A}}_{\mathrm{R}(\mc{Q}), \mathrm{N}(\mc{R})}^{(2)}\right)(i_{1},i_{2},i_{3})=\sum_{m=1}^{\min \{i_{2}, s\}}\displaystyle \sum_{\ell=1}^{s} \frac{\sum_{k=0}^{\overline{\alpha}_{q}} \overline{\alpha}_{t, k, u, i_{1},i_{2},i_{3},l, m} x^{k}}{\sum_{k=0}^{\underline{\alpha}_{q}} \underline{\alpha}_{t, k, u, i_{1},i_{2},i_{3},l, m} x^{k}},
\end{equation}
where  $\overline{\alpha}_{q}=\overline{\beta}+\overline{\eta}+\overline{\delta}$, $\underline{\alpha}_{q}=\underline{\beta}+\underline{\eta}+\underline{\delta}$, $
    \overline{\alpha}_{t, k, u, i_{1},i_{2},i_{3},l,m}=\sum_{t=0}^{k} \left(\sum_{u=0}^{k-t} \overline{q}_{t, i_{1},\ell,i_{3}} \overline{n}_{k-t-u, \ell,i_{2},i_{3}}\right)\overline{r}_{u,m,i_{2},i_{3}}$, and  $
    \underline{\alpha}_{t, k, u, i_{1},i_{2},i_{3},l,m}=\sum_{t=0}^{k} \left(\sum_{u=0}^{k-t} \underline{q}_{t, i_{1},\ell,i_{3}} \underline{n}_{k-t-u, \ell,i_{2},i_{3}}\right)\underline{r}_{u,m,i_{2},i_{3}}$. 
Simplifying the expression in equation \eqref{eq:elmentats2}, we get
\[
    \left(\Tilde{\mc{A}}_{\mathrm{R}(\mc{Q}), \mathrm{N}(\mc{R})}^{(2)}\right)(i_{1},i_{2},i_{3})=\frac{\overline{P}(i_{1},i_{2},i_{3})}{\underline{P}(i_{1},i_{2},i_{3})},
\]
where
\[
\underline{P}(i_{1},i_{2},i_{3})=\mathrm{PolynomialLCM}\left\{\sum_{k=0}^{\underline{\alpha}_{q}} \underline{\alpha}_{t, k, u, i_{1},i_{2},i_{3},\ell, m} x^{k} \mid m=\overline{1, \min \{i_{2}, s\}}, \ell=\overline{1, s}\right\}=\sum_{d=0}^{\underline{P}_{q}} \underline{P}_{ d, i_{1},i_{2},i_{3}} x^{d},\]
\[
\overline{P}(i_{1},i_{2},i_{3})=\sum_{m=1}^{\min \{i_{2}, s\}} \sum_{\ell=1}^{s}\left(P(t, k, u, i_{1},i_{2},i_{3},l, m) \sum_{k=0}^{\overline{\alpha}_{q}} \overline{\alpha}_{t, k, u, i_{1},i_{2},i_{3},l, m} x^{k}\right),
\]
and
\[
    P(t, k, u, i_{1},i_{2},i_{3},l, m)=\frac{\underline{P}(i_{1},i_{2},i_{3})}{ \sum_{k=0}^{\underline{\alpha}_{q}} \underline{\alpha}_{t, k, u, i_{1},i_{2},i_{3},l, m} x^{k}}=\sum_{d=0}^{\underline{P}_{q}-\underline{\alpha}_{q}} \gamma_{d,i_{1},i_{2},i_{3}} x^{d}.
\]
Therefore,
\[
    \overline{P}(i_{1},i_{2},i_{3})=\sum_{d=0}^{\underline{P}_{q}-\underline{\alpha}_{q}+\overline{\alpha}_{q}}\left(\sum_{k=1}^{\min\{i_{2}, s\}} \sum_{\ell=1}^s \sum_{t_{1}=0}^{d} \overline{\alpha}_{d,t, k, u, i_{1},i_{2},i_{3},l, m}P_{d-t_{1}, i_{1},i_{2},i_{3}}\right) x^{d},
\]
and completes the proof.
\end{proof}
In a special case, we can obtain the following representation for the Moore-Penrose inverse of tensors over the field of rational functions based on the $M$-$\mc{QDR}$ decomposition. 
\begin{theorem}\label{thm:qdrmpi}
Let $M\in \mbc(x)^{p\times p}$ and $\mc{A}\in \mbc(x)^{m\times n\times p}$ with entries are of the form
\[\mc{A}(i_{1},i_{2},i_{3})=\dfrac{\sum_{d=0}^{\overline{d}_{a}} \overline{a}_{d, i_{1},i_{2},i_{3}}x^{d}}{\sum_{d=0}^{\underline{d}_{a}} \underline{a}_{d, i_{1},i_{2},i_{3}} x^{d}},1\leq i_{1}\leq m, 1\leq i_{2}\leq n, 1\leq i_{3}\leq p.\]
%where $\overline{d}_{a}$ and $\underline{d}_{a}$ are maximal degrees of the numerators and denominators of the tensors $\mc{A}$ respectively. 
Consider $\mrk{\mc{A}}=\mathbf{s}$ and $\mc{A}=\mc{Q}*_M\mc{D}*_M\mc{R}$ be the  $M$-$\mc{QDR}$ decomposition of $\mc{A}$, where the respective entries of $\Tilde{\mc{Q}} \in \mbc(x)^{m \times s\times p}$ and $\Tilde{\mc{R}}\in \mbc(x)^{s \times n\times p}$, are of the form
\[
\Tilde{\mc{Q}}(i_{1},i_{2},i_{3})=\frac{\sum_{d=0}^{\overline{d}_{\mc{Q}}} \overline{q}_{d, i_{1},i_{2},i_{3}} x^{d}}{\sum_{d=0}^{\underline{d}_{\mc{Q}}} \underline{q}_{d, i_{1},i_{2},i_{3}} x^{d}}, \quad 1 \leq i_{1} \leq m, 1 \leq i_{2} \leq s, 1 \leq i_{3} \leq p,\]
\[\Tilde{\mc{R}}(i_{1},i_{2},i_{3})=\frac{\sum_{d=0}^{\overline{d}_{\mc{R}}} \overline{r}_{d, i_{1},i_{2},i_{3}} x^{d}}{\sum_{d=0}^{\underline{d}_{\mc{R}}} \underline{r}_{d, i_{1},i_{2},i_{3}} x^{d}}, \quad 1 \leq i_{1} \leq s, 1 \leq i_{2} \leq n,  i_{1} \leq j_{1}, 1 \leq i_{3} \leq p.\]
Assume that the entries of $\mc{N}=\left(\mc{Q}^{T}*_M\mc{A}*_M\mc{R}^{T}\right)^{-1} \in \mbc(x)^{s \times s\times p}$ be of the form
\begin{equation}\label{eq:nmp}
    \mc{N}(i_{1},i_{2},i_{3})=\frac{\sum_{d=0}^{\overline{d}_{\mc{N}}} \overline{n}_{d, i_{1},i_{2},i_{3}} x^{d}}{\sum_{d=0}^{\underline{d}_{\mc{N}}} \underline{n}_{d, i_{1},i_{2},i_{3}} x^{d}}.
\end{equation}
Consider $\underline{\theta}=\underline{d}_{\mc{R}}+\underline{d}_{\mc{N}}+\underline{d}_{\mc{Q}}$ and $\overline{\theta}=\overline{d}_{\mc{R}}+\overline{d}_{\mc{N}}+\overline{d}_{\mc{Q}}$. For $k=\overline{1,s}$ and $\ell=\overline{1,n}$ assume that 
\[
\overline{\alpha}_{t, i, k, l, j}=\sum_{t_{1}=0}^{t} \sum_{t_{2}=0}^{t-t_{1}} \overline{r}_{t_{1}, k, i}^{T} \overline{n}_{t-t_{1}-t_{2}, k, l} \overline{q}_{t_{2}, j, l}^{T}, \quad 0 \leq t \leq \overline{\theta},\]
\[
\underline{\alpha}_{t, i, k, l, j}=\sum_{t_{1}=0}^{t} \sum_{t_{2}=0}^{t-t_{1}} \underline{r}_{t_{1}, k, i}^{T} \underline{n}_{t-t_{1}-t_{2}, k, l} \underline{q}_{t_{2}, j, l}^{T}, \quad 0 \leq t \leq \underline{\theta},\]
\[
\underline{P}_{i_{1},i_{2},i_{3}}(x)=\mathrm{PolynomialLCM}\left\{\sum_{d=0}^{\underline{\theta}} \underline{\alpha}_{d,i_{1},i_{2},i_{3},k,\ell} x^{d} \mid k=\overline{1, s}, \ell=\overline{1, n}\right\}\]
with degree $\underline{d}_{P}$.  Then entries of $\mc{A}^{\dagger}$ are given by
\[\mc{A}^{\dg}(i_{1},i_{2},i_{3})=\frac{\overline{P}_{i_{1},i_{2},i_{3}}(x)}{\underline{P}_{i_{1},i_{2},i_{3}}(x)},\]
where
\[
\overline{P}_{i_{1},i_{2},i_{3}}(x)=\sum_{\ell=1}^{n} \sum_{k=1}^{s} \sum_{d=0}^{\underline{d}_{P}-\underline{\theta}+\overline{\theta}}\left(\sum_{t=0}^{d} \overline{\alpha}_{t, i_{1},i_{2},i_{3}, k, \ell} \beta_{d-t, i_{1},i_{2},i_{3},k,\ell}\right) x^{d},\]
%and $\beta_{d-t, i_{1},i_{2},i_{3},k,\ell}, 0 \leq t \leq \underline{d}_{P}-\underline{\theta}$, are coefficients of the polynomial 
$P_{i_{1},i_{2},i_{3},k,\ell}(x)=\frac{\underline{P}_{i_{1},i_{2},i_{3}}(x) }{\sum_{d=0}^{\underline{\theta}} \underline{\alpha}_{d,i_{1},i_{2},i_{3},k,\ell} x^{d} }$, for each $k=\overline{1, s}, \ell=\overline{1, n}$.
\end{theorem}

Since meromorphic functions are fractions of holomorphic functions, they can be extended as a nonzero polynomial in complex variables. We will compute $\mc{A}^{\dg}(z)$ from a random tensor $\mc{A}(z)$ whose entries are meromorphic functions,  like $e^z$, $\sin(z)$, $\sinh(z)$, $\cos(z)$, $\cosh(z)$ where $z\in\mb{C}$. More generally, when $\mc{A}(z)$ does not contain a holomorphic function, we can use Laurent formal power series to expand it as a truncated polynomial and can find the outer inverse of $\mc{A}(z)$.
\begin{example}\label{exm:ats2}
Consider two tensors $\mc{A}\in\mb{C}(z)^{2\times 2\times 2}$ and $\mc{W}\in\mb{C}(z)^{2\times 2\times 2}$ with frontal slices
 \begin{equation*}
     \mc{A}(:,:,1)=
\begin{pmatrix}
 \dfrac{2\sin z \sinh z}{e^z} & 0 \\
 0 & 0 \\
\end{pmatrix},~ \mc{A}(:,:,2)=\begin{pmatrix}
 0 & 0 \\
 2z & 0 \\
\end{pmatrix};~ \mc{W}(:,:,1)=
\begin{pmatrix}
 1+z & 0 \\
 0 & 0 \\
\end{pmatrix},~ \mc{W}(:,:,2)=\begin{pmatrix}
 0 & 0 \\
 z & 0 \\
\end{pmatrix}.
\end{equation*}
 and an invertible matrix $M=\begin{pmatrix}1&2\\1&1\end{pmatrix}$.
Approximate the tensor $\mc{A}$ using the Laurent expansion up to four-degree terms. 
 \begin{equation*}
     \mc{A}(:,:,1)\approx
\begin{pmatrix}
 \frac{z}{1+z+\frac{z^2}{2}+\frac{z^3}{6} }& 0 \\
 0 & 0 \\
\end{pmatrix} \mbox{ and }\mc{A}(:,:,2)=\begin{pmatrix}
 0 & 0 \\
 2 z& 0  \\
\end{pmatrix}.
 \end{equation*}
Then, $\mrk{\mc{A}}=(1,1)$ and $\mrk{\mc{W}}=(1,1)$. Then following Algorithm \ref{alg:outerM} one can compute $\mc{A}^{(2)}_{\mc{T},\mc{S}}$ as the following:
\begin{eqnarray*}
    \mc{A}^{(2)}_{\mc{T},\mc{S}}(:,:,1)=
\begin{pmatrix}
 \frac{z^3+3 z^2+6 z+6}{6 z} & 0 \\
 0 & 0 
\end{pmatrix},
\mc{A}^{(2)}_{\mc{T},\mc{S}}(:,:,2)=
\begin{pmatrix}
 0 & 0 \\
  \frac{z^3+3 z^2+6 z+6}{6 z+6} & 0 \\
\end{pmatrix}.
\end{eqnarray*}
\end{example}
\noindent We compute $\mc{A}^{\dagger}$ as a special case in the following example.

\begin{example}
Consider $\mc{A}, \mc{W},$ and $M$ as given in Example \ref{exm:ats2}. Using Algorithm \ref{alg:outerM} and Corollary \ref{co3.66}(i) we compute the Moore-penrose inverse of $\mc{A}$, where the entrees of $\mc{A}^\dagger$ are given by
{\tiny
\begin{eqnarray*}
  A^{\dagger}(1,1,1)&=& \frac{21 z^9+189 z^8+945 z^7+3213 z^6+7938 z^5+14742 z^4+20439 z^3+20493 z^2+13770 z+4698}{8 z^{13}+96 z^{12}+624 z^{11}+2784 z^{10}+9288 z^9+24192 z^8+50202 z^7+83484 z^6+110754 z^5+114912 z^4+89424 z^3+47952 z^2+13770 z}\\
  A^{\dagger}(1,2,1)&=& \frac{3 z^{12}+36 z^{11}+234 z^{10}+1044 z^9+3483 z^8+9072 z^7+18792 z^6+31104 z^5+40824 z^4+41472 z^3+31104 z^2+15552 z+3888}{4 z^{13}+48 z^{12}+312 z^{11}+1392 z^{10}+4644 z^9+12096 z^8+25101 z^7+41742 z^6+55377 z^5+57456 z^4+44712 z^3+23976 z^2+6885 z} \\
A^{\dagger}(1,1,2) &=&
 \frac{-9 z^9-81 z^8-405 z^7-1377 z^6-3402 z^5-6318 z^4-8748 z^3-8748 z^2-5832 z-1944}{8 z^{13}+96 z^{12}+624 z^{11}+2784 z^{10}+9288 z^9+24192 z^8+50202 z^7+83484 z^6+110754 z^5+114912 z^4+89424 z^3+47952 z^2+13770 z} \\
 A^{\dagger}(1,2,2) &=&\frac{-2 z^{12}-24 z^{11}-156 z^{10}-696 z^9-2322 z^8-6048 z^7-12519 z^6-20682 z^5-27027 z^4-27216 z^3-20088 z^2-9720 z-2268}{8 z^{13}+96 z^{12}+624 z^{11}+2784 z^{10}+9288 z^9+24192 z^8+50202 z^7+83484 z^6+110754 z^5+114912 z^4+89424 z^3+47952 z^2+13770 z} \\
 A^{\dagger}(2,1,1)&=& A^{\dagger}(2,2,1) ~=~ A^{\dagger}(2,2,1) ~=~  A^{\dagger}(2,1,2)~ =~ A^{\dagger}(2,1,2) ~=~ 0\\
\end{eqnarray*}
}
\end{example}

\section{Conclusion}\label{sec:con}
The tensor forms of the FRD and $M$-$\mc{QDR}$-decomposition were developed based on the $M$-product with powerful algorithms for their computations. The proposed $M$-$\mc{QDR}$-decomposition calculates the $M$-Moore-Penrose inverse and outer inverse of a given third-order tensor whose entries are either symbolic (polynomial in a single variable) or rational. The proposed $M$-$\mc{QDR}$-decomposition is demonstrated by solving a few examples and applied to investigate lossy color image compression.
 
%\section*{Acknowledgements} 

\section*{Data Availability Statement}
The data that generated and support the findings of this study are available from the corresponding author upon reasonable request.

\section*{Conflict of Interest}
The authors declare no potential conflict of interest.

\section*{Funding}
\begin{itemize}
    \item Ratikanta Behera is grateful for the support of the Anusandhan National Research Foundation (ANRF), Department of Science and Technology, India, under Grant No. EEQ/2022/001065. 
\item Jajati Keshari Sahoo is grateful for the support of the Anusandhan National Research Foundation (ANRF), Department of Science and Technology, India, under Grant No. SUR/2022/004357. 
\end{itemize}

\section*{ORCID}
Krushnachandra Panigrahy \orcidC \href{https://orcid.org/0000-0003-0067-9298}{ \hspace{2mm}\textcolor{lightblue}{https://orcid.org/0000-0003-0067-9298}} \\
Ratikanta Behera\orcidA \href{https://orcid.org/0000-0002-6237-5700}{ \hspace{2mm}\textcolor{lightblue}{ https://orcid.org/0000-0002-6237-5700}}\\
Jajati Keshari Sahoo \orcidC \href{https://orcid.org/0000-0001-6104-5171}{ \hspace{2mm}\textcolor{lightblue}{https://orcid.org/0000-0001-6104-5171}} \\
Biswarup Karmakar \orcidC \href{https://orcid.org/0009-0003-5635-5425}{ \hspace{2mm}\textcolor{lightblue}{https://orcid.org/0009-0003-5635-5425}} \\
\bibliographystyle{abbrv}
\bibliography{References}

\end{document}